\documentclass[10pt,reqno]{amsart}

\usepackage[dvipsnames]{xcolor}
\usepackage{amsmath,amssymb,amsthm,amsfonts,enumerate,tikz,bm}
\usepackage{mathrsfs}
\usetikzlibrary{matrix,arrows,positioning,calc}
\usepackage[all]{xy}
\usepackage[bookmarksnumbered,colorlinks]{hyperref}
\usepackage{url}
\numberwithin{equation}{section}
\usepackage{graphicx}
\usepackage[T1]{fontenc}
\usepackage{textcomp}
\usepackage{palatino,helvet}

\usepackage{footmisc}
\usepackage{float}
\usepackage{caption}

\newtheorem{definition}{Definition}[section]
\newtheorem{remark}[definition]{Remark}
\newtheorem{example}[definition]{Example}
\newtheorem{theorem}[definition]{Theorem}
\newtheorem{proposition}[definition]{Proposition}
\newtheorem{lemma}[definition]{Lemma}
\newtheorem{corollary}[definition]{Corollary}
\newtheorem{question}[definition]{Question}
\theoremstyle{remark}
\numberwithin{equation}{section}

\usepackage{paralist}

\usepackage{scrextend}
\deffootnote{2em}{0em}{\thefootnotemark\quad}

\newcommand{\Rmod}{{\rm Mod}(R)}
\newcommand{\modR}{{\rm Mod}(R^{\rm op})}
\newcommand{\Rch}{{\rm Ch}({\rm Mod}(R))}
\newcommand{\chR}{{\rm Ch}({\rm Mod}(R^{\rm op}))}

\newcommand{\pd}{{\rm pd}}
\newcommand{\id}{{\rm id}}
\newcommand{\fd}{{\rm fd}}

\newcommand{\Hom}{{\rm Hom}}
\newcommand{\Ext}{{\rm Ext}}
\newcommand{\Tor}{{\rm Tor}}

\makeatletter
\def\@seccntformat#1{%
  \protect\textup{\protect\@secnumfont
    \ifnum\pdfstrcmp{section}{#1}=0 \scshape\bfseries\fi% section # in \scshape and \bfseries
    \ifnum\pdfstrcmp{subsection}{#1}=0 \bfseries\fi% subsection # in \bfseries
    \csname the#1\endcsname
    \protect\@secnumpunct
  }%
}
\makeatother

\begin{document}

\title{$\bm{n}$-Absolutely pure and $\bm{n}$-flat modules, revisited}

\author{Rafael Parra}
\address[R. Parra]{Instituto de Matem\'atica y Estad\'istica ``Prof. Ing. Rafael Laguardia''. Facultad de Ingenier\'ia. Universidad de la Rep\'ublica. CP11300. Montevideo, URUGUAY}
\email{rparra@fing.edu.uy}

\author{Marco A. P\'erez}
\address[M. A. P\'erez]{Instituto de Matem\'atica y Estad\'istica ``Prof. Ing. Rafael Laguardia''. Facultad de Ingenier\'ia. Universidad de la Rep\'ublica. CP11300. Montevideo, URUGUAY}
\email{mperez@fing.edu.uy}

\maketitle

\begin{abstract}
We revisit the concepts of $n$-absolutely pure and $n$-flat modules, and their relation with $n$-coherent rings in the sense of Lee. We provide new proofs for some known claims about these modules and rings, and show other new properties. We also give some corrections to the existing literature.  
%In the paper ``$n$-Flat and $n$-FP-injective modules'', the authors claim that an $R$-module $M$ is $n$-FP-injective if and only if $\Ext^1_R(R/I,M) = 0$ for any finitely generated left ideal $I$ with projective dimension at most $n-1$. We point out a mistake in their proof. On the other hand, we are not aware if this assertion is true, but we can at least assert its validity in the case where $R$ is left coherent. Other mistakes in this work are also detected and fixed. 
\end{abstract}

%\setcounter{tocdepth}{2}
%\tableofcontents

\pagestyle{myheadings}
\markboth{\rightline {\scriptsize R. Parra and M. A. P\'{e}rez}}
         {\leftline{\scriptsize $n$-Absolutely pure and $n$-flat modules, revisited}}

%%%%%%%%%%%%%%%%%%%%%%%%%%%%%%%%%%%%
%%%%%%%%%%%%%%%%%%%%%%%%%%%%%%%%%%%%
%%%%%%%%%%%%%%%%%%%%%%%%%%%%%%%%%%%%
%%%%%%%%%%%%%%%%%%%%%%%%%%%%%%%%%%%%

\setcounter{tocdepth}{1}
\tableofcontents

%%%%%%%%%%%%%%%%%%%%%%%%%%%%%%%%%%%%%
%%%%%%%%%%%%%%%%%%%%%%%%%%%%%%%%%%%%%
%%%%%%%%%%%%%%%%%%%%%%%%%%%%%%%%%%%%%
%%%%%%%%%%%%%%%%%%%%%%%%%%%%%%%%%%%%%

\section*{Introduction}         

Among the existing generalizations of coherent rings, one has the family of left $n$-coherent rings introduced by Lee in \cite{Lee}. These are defined as those rings $R$ over which if $K$ is a finitely generated submodule of a finitely generated and free left $R$-module with projective dimension $\leq n-1$, then $K$ is finitely presented. Here, $n$ is a positive integer or $n = \infty$. In particular, every ring is left $1$-coherent, and left coherent rings coincide with left $\infty$-coherent rings. 

The roles of absolutely pure (a.k.a. FP-injective) left $R$-modules and flat right $R$-modules are very well known in the study of left coherent rings. The absolutely pure and flat analogs in the study of left $n$-coherent rings are the $n$-absolutely pure and $n$-flat modules, also introduced in \cite{Lee}. The former are defined as the right orthogonal complement of $\mathsf{FP}^{\leq n}$ under Ext, and the latter as the left orthogonal complement of $\mathsf{FP}^{\leq n}$ under Tor, where $\mathsf{FP}^{\leq n}$ denotes the class of finitely presented left $R$-modules with projective dimension $\leq n$. These classes of modules, denoted by $\mathsf{A}_n$ and $\mathsf{F}_n$, are useful to provide equivalent descriptions of left $n$-coherent rings. For instance, it is proved in \cite{Lee} that $R$ is $n$-coherent if, and only if, direct products of $n$-flat right $R$-modules are $n$-flat. This equivalence is further extended by Yang and Liu in \cite{YangLiu11}. Specifically, they show that the following assertions are equivalent:
\begin{itemize}
\item $R$ is left $n$-coherent;

\item $\mathsf{F}_n$ is preenveloping;

\item $\mathsf{A}_n$ is closed under direct limits;

\item $\mathsf{A}_n$ is coresolving.
\end{itemize}

In the present paper we revisit the concepts of $n$-coherent rings, $n$-absolutely pure and $n$-flat modules, and the relation between them. In reviewing the literature, we note that certain arguments in some proofs of \cite{Lee,YangLiu11} are incomplete. Our main objective is to provide alternative proofs for most of the affected results in these references, thereby ensuring their validity. In particular, we confirm that the equivalent descriptions of left $n$-coherent rings mentioned above remain correct. By employing relatively recent techniques from the theories of duality pairs and cotorsion pairs, we not only recover these important results on $n$-coherency but also extend several of them. Our main contribution is an extension of the characterization of $n$-coherent rings, by adding more properties for the classes $\mathsf{FP}^{\leq n}$, $\mathsf{A}_n$ and $\mathsf{F}_n$ of modules over a ring $R$ that hold if, and only if, $R$ is $n$-coherent. Some of our findings are also related to a couple of open questions in the field, that lead us to define a new subclass of $n$-coherent rings and expansions of the classes $\mathsf{A}_n$ and $\mathsf{F}_n$, which turn out to be of interest by themselves. Concerning examples, we show some methods that allow us to generate more instances of $n$-absolutely pure and $n$-flat modules and chain complexes.

%%%%%%%%%%%%%%%%%%%%%%%%%%%%%%%%%%%%%
%%%%%%%%%%%%%%%%%%%%%%%%%%%%%%%%%%%%%
%%%%%%%%%%%%%%%%%%%%%%%%%%%%%%%%%%%%%
%%%%%%%%%%%%%%%%%%%%%%%%%%%%%%%%%%%%%

\subsection*{Organization}       

In Section \ref{sec:prelims} we recall the necessary preliminary notions on coherent rings, cotorsion pairs and duality pairs. 

Section \ref{sec:general_properties} is devoted to display the general properties of the classes $\mathsf{A}_n$ and $\mathsf{F}_n$. Specifically, in Proposition \ref{prop:basic_properties} we show that $(\mathsf{F}_n,\mathsf{A}_n)$ is always a perfect duality pair. In particular, this implies that $(\mathsf{F}_n,[\mathsf{F}_n]^{\perp_1})$ is a perfect cotorsion pair, and so that every right $R$-module has an $n$-flat cover. The class $\mathsf{A}_n$ is also the right half of a complete cotorsion pair, which yields the existence of special $n$-absolutely pure preenvelopes. We also define the classes of Baer $n$-absolutely pure and Baer $n$-flat modules as expansions of $\mathsf{A}_n$ and $\mathsf{F}_n$. These new classes have the same general properties of $\mathsf{A}_n$ and $\mathsf{F}_n$, and will be important to study the subclass of left $n$-coherent rings given by those rings $R$ in which every finitely generated left ideal with projective dimension $\leq n-1$ is finitely presented. These rings will be called \emph{weakly} left $n$-coherent. 

In Section \ref{sec:relations} we explore the relation between left $n$-coherent rings, $n$-absolutely pure left $R$-modules and $n$-flat right $R$-modules. Our main result is Theorem \ref{theo1}, where we extend the characterizations of $n$-coherent rings given in \cite{Lee,YangLiu11}. We give alternative proofs to the arguments in \cite{Lee,YangLiu11}, and show new characterizations of $n$-coherency. For instance, we prove that $R$ is left $n$-coherent if, and only if, $\mathsf{A}_n$ is $n$-tilting (or equivalently, $\mathsf{F}_n$ is $n$-cotilting). 

A characterization of weakly $n$-coherent rings analogous to Theorem \ref{theo1} is stated and proved in Section \ref{sec:ideals}, and certain differences between these two characterizations will be worth mentioning. On the other hand, we address the question of whether $n$-coherent and weakly $n$-coherent rings are the same class of rings. Indeed, this is the case for $n = \infty$. We give a partial answer to this question in Corollary \ref{cor:when_weak_is_strong}. This result, on the other hand, is a consequence of Proposition \ref{prop:An_vs_IAn}, which is related to an open question posed in \cite{Shen11}, where the author asserts that it is unknown whether a right $f$-injective ring $R$ is right FP-injective or not.

The existence of covers and preenvelopes by the classes of (Baer) $n$-absolutely pure and (Baer) $n$-flat modules is further explored in Section \ref{sec:approximations}. We show for instance in Proposition \ref{prop:ump} that the existence of (Baer) $n$-absolutely pure covers with the unique mapping property for any $R$-module is equivalent to the fact that the class of (Baer) $n$-absolutely pure $R$-modules has cokernels. Another noteworthy property is that the existence of monic $n$-absolutely pure covers only holds for $n = 1$ (see Proposition \ref{prop:injectives_closed_under_quotients}). 

Finally, Section \ref{sec:examples} is devoted to the construction of new $n$-absolutely pure and $n$-flat modules from old ones. For example, we give sufficient conditions that a ring homomorphism $R \to S$ needs to fulfill so that $n$-absolutely purity and $n$-flatness can be transferred in certain cases between $R$-modules and $S$-modules via extension, restriction and coextension of scalars. We also achieve some descriptions of $n$-absolutely pure and $n$-flat modules over triangular matrix rings. In the last part of this section, we propose the concept of $n$-absolutely pure and $n$-flat chain complexes, and show that a chain complex is $n$-absolutely pure (resp., $n$-flat) if, and only if, all of its module components are $n$-absolutely pure (resp., $n$-flat).

%%%%%%%%%%%%%%%%%%%%%%%%%%%%%%%%%%%%%
%%%%%%%%%%%%%%%%%%%%%%%%%%%%%%%%%%%%%
%%%%%%%%%%%%%%%%%%%%%%%%%%%%%%%%%%%%%
%%%%%%%%%%%%%%%%%%%%%%%%%%%%%%%%%%%%%

\section{Preliminaries}\label{sec:prelims}

Given an associative ring $R$ with identity, we denote by ${\rm Mod}(R)$ the category of (unital) left $R$-modules. By ``$R$-modules'' we shall mean ``left $R$-modules'', while right $R$-modules will be $R^{\rm op}$-modules, where $R^{\rm op}$ denotes the opposite ring. The category of right $R$-modules will be denoted by ${\rm Mod}(R^{\rm op})$. The \emph{projective}, \emph{injective} and \emph{flat} dimensions of a left (or right) $R$-module $M$ are denoted by $\pd(M)$, $\id(M)$ and $\fd(M)$, respectively. The \emph{Pontryagin dual} or \emph{character module} of $M$ is defined by $M^+ := \Hom_{\mathbb{Z}}(M,\mathbb{Q/Z})$. This construction defines an exact contravariant functor $(-)^+ \colon {\rm Mod}(R) \to {\rm Mod}(R^{\rm op})$ (or $(-)^+ \colon {\rm Mod}(R^{\rm op}) \to {\rm Mod}(R)$).

Lee in \cite{Lee} introduced the following class of rings.

\begin{definition}\label{def:Lee-n-coherent}
Given $n$ a positive integer or $n = \infty$, a ring $R$ is \textbf{(left) $\bm{n}$-coherent} if every finitely generated submodule of a finitely generated and free $R$-module with projective dimension at most $n-1$ is finitely presented. \\
\end{definition}

The reader must we warned that there is a namesake concept defined in \cite{Costa94}. According to Costa, a ring $R$ is (left) $n$-coherent\footnote{Costa defines this notion for commutative rings, but this assumption can be dropped if one is interested in general properties and relative homological algebra associated to these rings (see for instance \cite{BP17}).} if every $R$-module with a finite $n$-presentation has a finite $(n+1)$-presentation. Given $M \in \Rmod$, a \emph{finite $n$-presentation} of $M$ is an exact sequence
\[
P_n \to P_{n-1} \to \cdots \to P_1 \to P_0 \to M \to 0
\]
where each $P_k$ is finitely generated and free. From now on, $n$-coherency will be considered in the sense of \cite{Lee} (and the rings satisfying Definition \ref{def:Lee-n-coherent} will be called \emph{Lee $n$-coherent}), although in some occasions we shall consider the class $\mathsf{FP}_2$ of $R$-modules admitting a finite $2$-presentation, as in the following remark.

\begin{remark}\label{rmk:FP2}
A ring $R$ is a Lee $n$-coherent if, and only if, every finitely presented $R$-module with projective dimension $\leq n$ belongs to $\mathsf{FP}_2$. This can be deduced directly from Definition \ref{def:Lee-n-coherent} and \cite[Thm. 2.1.2]{Glaz89}. \\
\end{remark}

\begin{example} ~\
\begin{enumerate}[(1)]
\item Every ring is Lee $1$-coherent. 

\item Coherent rings are exactly those which are Lee $n$-coherent, where $n$ is the (left) global dimension of $R$. 

\item Every Lee $\infty$-coherent ring is left coherent.

\item Let $k$ be a field, and consider the local ring $R := k[x_1,x_2,\dots] / \mathfrak{m}^2$ where $\mathfrak{m}$ is the maximal ideal $\mathfrak{m} = (x_1, x_2, \dots)$. From \cite[Ex. 2.2]{ChristensenEstradaThompson26}, we have that $R$ is a non coherent ring over which non projective $R$-modules have infinite projective dimension. Then, every finitely presented $R$-module with projective dimension $\leq n$ (with $n < \infty$) must be projective, and hence in $\mathsf{FP}_2$. By the previous remark, we have that $R$ is Lee $n$-coherent for every positive integer $n$ (with $n < \infty$). 
\end{enumerate}
\end{example}

From these examples we have that, in particular, $n$-coherency (with $n$ finite) does not necessarily imply coherency. 

For the purposes of this paper, the following class of rings will be also important.

\begin{definition}\label{def:weakly-Lee-n-coherent}
Given $n$ a positive integer or $n = \infty$, a ring $R$ is \textbf{weakly Lee (left) $\bm{n}$-coherent} if every finitely generated (left) ideal with projective dimension at most $n-1$ is finitely presented.
\end{definition}

%%%%%%%%%%%%%%%%%%%%%%%%%%%%%%%%%%%%
%%%%%%%%%%%%%%%%%%%%%%%%%%%%%%%%%%%%

\subsection*{Cotorsion pairs}

Two classes $\mathcal{A,B} \subseteq {\rm Mod}(R)$ form a \emph{cotorsion pair} $(\mathcal{A,B})$ if
\begin{align*}
\mathcal{A} & = {}^{\perp_1}\mathcal{B} := \{ M \in {\rm Mod}(R) \ \text{:} \ \text{Ext}^1_R(M,B) = 0 \text{ for every } B \in \mathcal{B} \}, \\
\mathcal{B} & = \mathcal{A}^{\perp_1} := \{ N \in {\rm Mod}(R) \ \text{:} \ \text{Ext}^1_R(A,N) = 0 \text{ for every } A \in \mathcal{A} \}.
\end{align*}
Moreover, if $\text{Ext}^i(A,B) = 0$ for every $A \in \mathcal{A}$, $B \in \mathcal{B}$ and $i \geq 1$, then the cotorsion pair $(\mathcal{A,B})$ is said to be \emph{hereditary}. This is equivalent to saying that $\mathcal{A}$ is closed under kernels of epimorphisms between objects in $\mathcal{A}$ (or that $\mathcal{A}$ is closed under epikernels, for short), or that $\mathcal{B}$ is closed under monocokernels. %If $\mathsf{Inj}(R)$ and $\mathsf{Proj}(R)$ denote the classes of injective and projective $R$-modules
Furthermore, recall that a class is \emph{resolving} (resp., \emph{coresolving}) if it is closed under extensions, epikernels and contains the class of projective $R$-modules (resp., closed under extensions, monocokernels and contains the class of injective $R$-modules). Then, it is clear that a cotorsion pair $(\mathcal{A,B})$ is hereditary if, and only if, $\mathcal{A}$ is resolving (or equivalently, $\mathcal{B}$ is coresolving). 

A cotorsion pair $(\mathcal{A,B})$ is \emph{cogenerated} by $\mathcal{S} \subseteq {\rm Mod}(R)$ if $\mathcal{B} = \mathcal{S}^{\perp_1}$. If $\mathcal{S}$ is a set, then Eklof and Trlifaj's \cite[Thm. 10]{ET01} asserts that $(\mathcal{A,B})$ is \emph{complete}, which means that for every $M \in {\rm Mod}(R)$ there are exact sequences $0 \to B' \to A \to M \to 0$ and $0 \to M \to B \to A' \to 0$ with $A, A' \in \mathcal{A}$ and $B, B' \in \mathcal{B}$. The morphisms $A \to M$ and $M \to B$ are particular examples of approximations. Indeed, given a class $\mathcal{X} \subseteq {\rm Mod}(R)$ and $M \in {\rm Mod}(R)$, an $R$-homomorphism $\varphi \colon X \to M$ with $X \in \mathcal{X}$ is an \emph{$\mathcal{X}$-precover} (or a \emph{right $\mathcal{X}$-approximation}) of $M$ if for every $R$-homomorphism $\varphi' \colon X' \to M$ with $X' \in \mathcal{X}$ there exists an $R$-homomorphism $h \colon X' \to X$ such that $\varphi' = \varphi \circ h$. An $\mathcal{X}$-precover $\varphi \colon X \to M$ is an \emph{$\mathcal{X}$-cover} if for the case $X' = X$ and $\varphi' = \varphi$, the equality $\varphi = \varphi \circ h$ can only be completed by automorphisms of $X$. Finally, an $\mathcal{X}$-precover $\varphi \colon X \to M$ is \emph{special} if $\varphi$ is epic and ${\rm Ker}(\varphi) \in \mathcal{X}^{\perp_1}$. The concepts of (\emph{special}) (\emph{pre})\emph{envelopes} are dual.

%%%%%%%%%%%%%%%%%%%%%%%%%%%%%%%%%%%%
%%%%%%%%%%%%%%%%%%%%%%%%%%%%%%%%%%%%

\subsection*{Duality pairs}

These pairs were introduced by Holm and J{\o}rgensen in \cite{HJ09}. Concretely, two classes $\mathcal{M} \subseteq {\rm Mod}(R)$ and $\mathcal{C} \subseteq {\rm Mod}(R^{\rm op})$ form a \emph{duality pair} $(\mathcal{M,C})$ \emph{over $R$} if:
\begin{enumerate}
\item $M \in \mathcal{M}$ if, and only if, $M^+ \in \mathcal{C}$. 

\item $\mathcal{C}$ is closed under direct summands and finite coproducts. 
\end{enumerate}
Moreover, by \cite[Thm. 3.1]{HJ09} $\mathcal{M}$ is closed under pure submodules, pure quotients and pure extensions. One has a similar notion of duality pair in the case $\mathcal{M}$ is a class of $R^\textrm{op}$-modules, and $\mathcal{C}$ a class of $R$-modules. In this case one says that $(\mathcal{M,C})$ is a duality pair over $R^{\rm op}$. A duality pair $(\mathcal{M,C})$ over $R$ is called:
\begin{itemize}
\item \emph{(co)product closed} if $\mathcal{M}$ is closed under (co)products. 

\item \emph{perfect} if it is coproduct closed, $\mathcal{M}$ is closed under extensions and contains $R$ (regarded as an $R$-module). 
\end{itemize}
The previous types of duality pairs are important since they induce approximations by the classes $\mathcal{M}$ and $\mathcal{C}$. Indeed, it was proven in \cite[Thm. 3.1]{HJ09} that:
\begin{itemize}
\item If $(\mathcal{M,C})$ is coproduct closed, then $\mathcal{M}$ is covering.

\item If $(\mathcal{M,C})$ is product closed, then $\mathcal{M}$ is preenveloping.

\item If $(\mathcal{M,C})$ is perfect, then $(\mathcal{M},\mathcal{M}^{\perp_1})$ is a \emph{perfect cotorsion pair} (that is, a cotorsion pair such that $\mathcal{M}$ is covering and $\mathcal{M}^{\perp_1}$ is enveloping).
\end{itemize}
Duality pairs were also studied by Bravo, Hovey and Gillespie in \cite{BGH14,GillespieDuality}. We point out \cite[Prop. 2.3]{GillespieDuality}, which asserts that the containments $\mathsf{P} \subseteq \mathcal{M}$ and $\mathsf{I} \subseteq \mathcal{C}$ hold for every perfect duality pair $(\mathcal{M,C})$ over $R$, where $\mathsf{P}$ and $\mathsf{I}$ denote the classes of projective $R$-modules and injective $R^\textrm{op}$-modules, respectively. Furthermore, for such pairs the class $\mathcal{M}$ is closed under direct limits and so the class $\mathsf{F}$ of flat $R$-modules is contained in $\mathcal{M}$ (by Lazard-Govorov's Theorem \cite[Thm. 4.34]{Lam99}). 

In \cite[Appx. A]{BGH14} and \cite[Def. 2.4]{GillespieDuality}, a duality pair $(\mathcal{M,C})$ over $R$ is called \emph{symmetric} if $(\mathcal{C,M})$ is a duality pair over $R^{\rm op}$. A \emph{complete} duality pair is a symmetric duality pair $(\mathcal{M,C})$ such that $(\mathcal{M,C})$ is a perfect duality pair.

%%%%%%%%%%%%%%%%%%%%%%%%%%%%%%%%%%%%%
%%%%%%%%%%%%%%%%%%%%%%%%%%%%%%%%%%%%%
%%%%%%%%%%%%%%%%%%%%%%%%%%%%%%%%%%%%%
%%%%%%%%%%%%%%%%%%%%%%%%%%%%%%%%%%%%%

\section{General properties of $n$-absolutely pure and $n$-flat modules}\label{sec:general_properties}

In what follows, we denote by $\mathsf{FP}^{\leq n}$ the class of finitely presented $R$-modules with projective dimension at most $n \geq 1$. Thus, the classes of $n$-absolutely pure $R$-modules and $n$-flat $R^{\rm op}$-modules (introduced in \cite{Lee}), denoted by $\mathsf{A}_n$ and $\mathsf{F}_n$, respectively, are given by:
\begin{align*}
\mathsf{A}_n & = (\mathsf{FP}^{\leq n})^{\perp_1} = \{ M \in {\rm Mod}(R) \ \text{ : } \ \Ext^1_R(F,M) = 0 \text{ for every } F \in \mathsf{FP}^{\leq n} \}, \\
\mathsf{F}_n & = {}^{\top_1}(\mathsf{FP}^{\leq n}) = \{ N \in {\rm Mod}(R^{\rm op}) \ \text{ : } \ \Tor^1_R(N,F) = 0 \text{ for every } F \in \mathsf{FP}^{\leq n} \}. \\
\end{align*}

\begin{example} ~\
\begin{enumerate}[(1)]
\item Injective and FP-injective (a.k.a. absolutely pure) $R$-modules are $n$-absolutely pure. Similarly, projective and flat $R^{\rm op}$-modules are $n$-flat.

\item FP-injective $R$-modules coincide with the $\infty$-absolutely pure $R$-modules, and flat $R^{\rm op}$-modules coincide with $\infty$-flat $R$-modules.

\item If R is a commutative domain, then an $R$-module is $1$-flat (resp., $1$-absolutely pure) if and only if it is torsion free (resp., divisible). See \cite[Lems. 1 \& 3]{Lee}.
\end{enumerate}
\end{example}

In \cite[Lem. 5]{Lee}, the author shows that $N \in \mathsf{F}_n$ if, and only if, $N^+ \in \mathsf{A}_n$. Moreover, it is clear from their definitions that $\mathsf{A}_n$ is closed under isomorphisms, direct summands, pure submodules, and finite coproducts, and that $\mathsf{F}_n$ is closed under isomorphisms, coproducts, extensions and that $R \in \mathsf{F}_n$. Hence, we have the following.

\begin{proposition}\label{prop:basic_properties}
The following assertions hold:
\begin{enumerate}[(1)]
\item $(\mathsf{F}_n,\mathsf{A}_n)$ is a perfect duality pair. 
\item $(\mathsf{F}_n,[\mathsf{F}_n]^{\perp_1})$ is a perfect cotorsion pair. In particular, every $R^{\rm op}$-module has an $n$-flat cover. 
\item $\mathsf{F}_n$ is closed under pure submodules, pure quotients and pure extensions. 
\item $({}^{\perp_1}[\mathsf{A}_n],\mathsf{A}_n)$ is a cotorsion pair cogenerated by a set. In particular, every $R$-module has a $n$-absolutely pure special preenvelope. 
\item $\mathsf{A}_n$ is closed under arbitrary coproducts. 

\item ${}^{\perp_1}[\mathsf{A}_n] \subseteq \mathsf{P}^{\leq n}$, where $\mathsf{P}^{\leq n}$ denotes the class of $R$-modules with projective dimension at most $n$. 

\item $[\mathsf{F}_n]^{\perp_1} \subseteq \mathsf{I}^{\leq n}$, where $\mathsf{I}^{\leq n}$ denotes the class of $R^{\rm op}$-modules with injective dimension at most $n$. \\
\end{enumerate}
\end{proposition}

\begin{proof} ~\
\begin{itemize}
\item Parts (1), (2) and (3) follow from the previous comments and the facts on duality pairs mentioned in the preliminaries. 

\item For part (4), in the equality $\mathsf{A}_n = (\mathsf{FP}^{\leq n})^{\perp_1}$ note that $\mathsf{FP}^{\leq n}$ is a subclass of the class of finitely generated $R$-modules, and so it is skeletally small. Eklof and Trlifaj theorem \cite[Thm. 10]{ET01} implies then that the cotorsion pair $({}^{\perp_1}[\mathsf{A}_n],\mathsf{A}_n)$ is complete. 

\item Part (5) follows by \cite[Prop. 2.10]{BOPP22}. 

\item For part (6), the containment $\mathsf{FP}^{\leq n} \subseteq \mathsf{P}^{\leq n}$ implies that ${}^{\perp_1}[\mathsf{A}_n] \subseteq \mathsf{P}^{\leq n}$ (see \cite[comments right after Def. 7.1.1 \& Thm. 7.4.6]{EJ00}).

\item Finally, for part (7), let $\mathsf{F}^{\leq n}$ denote the class of $R^{\rm op}$-modules with flat dimension at most $n$. Thus, note that 
\[
\ \ \ \ \ \ \ \ \ \ \ \ \mathsf{FP}^{\leq n} \subseteq \mathsf{F}^{\leq n} \ \Rightarrow \ {}^{\top_1}[\mathsf{F}^{\leq n}] \subseteq {}^{\top_1}[\mathsf{FP}^{\leq n}] = \mathsf{F}_n \ \Rightarrow \ [\mathsf{F}_n]^{\perp_1} \subseteq ({}^{\top_1}[\mathsf{F}^{\leq n}])^{\perp_1}. 
\]
It suffices to show that $({}^{\top_1}[\mathsf{F}^{\leq n}])^{\perp_1} \subseteq \mathsf{I}^{\leq n}$. So let $L \in ({}^{\top_1}[\mathsf{F}^{\leq n}])^{\perp_1}$ and $U \in \modR$. Denote by $\Omega^n(U)$ any projective $n$-syzygy of $U$, and consider $V \in \mathsf{F}^{\leq n}$. We have that
\[
\Tor^R_1(\Omega^n(U),V) \cong \Tor^R_{n+1}(U,V) = 0,
\] 
that is, $\Omega^n(U) \in {}^{\top_1}[\mathsf{F}^{\leq n}]$, and then
\[
\Ext^{n+1}_R(U,L) \cong \Ext^1_R(\Omega^n(U),L) = 0.
\]
Hence, ${\rm id}(L) \leq n$. 
\end{itemize}
\end{proof}

\begin{remark}\label{obs-initial} ~\
\begin{enumerate}
\item Part (3) of the previous proposition was originally proved in \cite[Prop. 2.1]{YangLiu11}. Our argument provides a shorter proof for this fact.

\item Part (4) was firstly claimed in \cite[Thm. 2.1 (2)]{YangLiu11}. Our argument fixes a gap we have found in the proof provided in  \cite{YangLiu11}, since it depends on \cite[Lem. 2.1 (2)]{YangLiu11}. This result, on the other hand, has also a gap in its proof (see Remark \ref{rmk0} below). 
\end{enumerate}
\end{remark}

It will be of interest to consider the subclass of $\mathsf{FP}^{\leq n}$ formed by its cyclic $R$-modules. Namely, consider the set 
\[
\mathsf{CFP}^{\leq n} := \{ R / I \ \text{ : } \ I \lhd R \text{ is a finitely generated left ideal of $R$ with } \text{pd}(I) \leq n-1 \},
\]
along with the Ext and Tor orthogonal complements\footnote{Not to be confused with \cite{Udhayakumar17}, where the author considers Ext and Tor orthogonal complements of the set $$\{ R / I \ \text{ : } \ I \lhd R \text{ is a finitely presented left ideal of $R$ with } \text{pd}(I) \leq n-1 \},$$ which are called $n$-IFP-injective and $n$-IFP-flat modules.} of $\mathsf{CFP}^{\leq n}$:
\begin{align*}
\mathsf{BA}_n & = (\mathsf{CFP}^{\leq n})^{\perp_1} & & \text{and} & \mathsf{BF}_n & = {}^{\top_1}(\mathsf{CFP}^{\leq n}).
\end{align*}
Modules in $\mathsf{BA}_n$ may be referred to as \emph{Baer $n$-absolutely pure $R$-modules}, while those in $\mathsf{BF}_n$ as \emph{Baer $n$-flat $R^{\rm op}$-modules}. It is clear that $\mathsf{BA}_n$ is closed under isomorphisms, direct summands, pure submodules, and finite coproducts, and that $\mathsf{BF}_n$ is closed under isomorphisms, coproducts, extensions and that $R \in \mathsf{BF}_n$. We have the following analog of Proposition \ref{prop:basic_properties}, which follows by a similar argument.

\begin{proposition}\label{prop:basic_properties_cyclic}
The following assertions hold:
\begin{enumerate}[(1)]
\item $(\mathsf{BF}_n,\mathsf{BA}_n)$ is a perfect duality pair. 

\item $(\mathsf{BF}_n,[\mathsf{BF}_n]^{\perp_1})$ is a perfect cotorsion pair. 

\item $\mathsf{BF}_n$ is closed under pure submodules, pure quotients and pure extensions. 

\item $({}^{\perp_1}[\mathsf{BA}_n],\mathsf{BA}_n)$ is a cotorsion pair cogenerated by a set.

\item $\mathsf{BA}_n$ is closed under arbitrary coproducts. 

\item ${}^{\perp_1}[\mathsf{BA}_n] \subseteq \mathsf{P}^{\leq n}$. 

\item $[\mathsf{BF}_n]^{\perp_1} \subseteq \mathsf{I}^{\leq n}$. \\
\end{enumerate}
\end{proposition}

%\begin{proof}
%It is clear from their definitions that $\mathsf{BA}_n$ is closed under direct summands and finite direct sums, and that $\mathsf{BF}_n$ is closed under coproducts, extensions and that $R \in \mathsf{BF}_n$. Moreover, note that $N \in \mathsf{BF}_n$ if, and only if, $N^+ \in \mathsf{BA}_n$, since there is a natural isomorphism $\Ext^1_R(R/I, N^+) \cong [\Tor^R_1(N, R/I)]^+$ by \cite[Thm. 3.2.1]{EJ00}. Thus, we get parts (1) and (2). Part (4), on the other hand, is immediate
%\end{proof}

It is clear that the containments $\mathsf{A}_n \subseteq \mathsf{BA}_n$ and $\mathsf{F}_n \subseteq \mathsf{BF}_n$ always hold. We are not aware if the converse containments $\mathsf{A}_n \supseteq \mathsf{BA}_n$ and $\mathsf{F}_n \supseteq \mathsf{BF}_n$ hold in general. We shall comment more on this in Proposition \ref{prop:An_vs_IAn} and Corollary \ref{cor:when_weak_is_strong}.

%%%%%%%%%%%%%%%%%%%%%%%%%%%%%%%%%%%%%
%%%%%%%%%%%%%%%%%%%%%%%%%%%%%%%%%%%%%
%%%%%%%%%%%%%%%%%%%%%%%%%%%%%%%%%%%%%
%%%%%%%%%%%%%%%%%%%%%%%%%%%%%%%%%%%%%

\section{Relation between $n$-absolutely pure, $n$-flat modules and $n$-coherent rings}\label{sec:relations}

Under the assumption that the ground ring $R$ is Lee $n$-coherent, we can obtain more properties for the classes $\mathsf{F}_n$ and $\mathsf{A}_n$. Moreover, some of these properties are actually equivalent to the $n$-coherency of $R$. We start with the characterization of Lee $n$-coherent rings stated below. It will be convenient to recall that an $R$-module $T$ is \emph{$n$-tilting} if the following assertions hold:
\begin{itemize}
\item ${\rm pd}(T) \leq n$.
\item $\Ext^i_R(T,T^{(I)}) = 0$, for every $i \geq 1$ and every index set $I$, where $T^{(I)}$ denotes the coproduct of copies of $T$ indexed by $I$. 
\item There exist $r \geq 0$ and an exact sequence $0 \to R \to T_0 \to T_1 \to \cdots \to T_r \to 0$ such that each $T_i$ is isomorphic to a direct summand of a coproduct of copies of $T$.
\end{itemize}
Classes of $R$-modules of the form 
\[
T^\perp := \{ N \in \Rmod \, \text{:} \, \Ext^i_R(T,N) = 0 \, \text{for every} \, i > 0 \}
\]
are called \emph{$n$-tilting classes} if $T$ is an $n$-tilting $R$-module. Dually, an $R^{\rm op}$-module $C$ is \emph{$n$-cotilting} if the following assertions hold:
\begin{itemize}
\item ${\rm id}(C) \leq n$.
\item $\Ext^i_R(C^{I},C) = 0$, for every $i \geq 1$ and every index set $I$, where $C^{I}$ denotes the product of copies of $C$ indexed by $I$. 
\item There exist $r \geq 0$, an injective cogenerator $Q$ of $\modR$, and an exact sequence $0 \to C_r \to \cdots \to C_1 \to C_0 \to Q \to 0$ such that each $C_i$ is isomorphic to a direct summand of a product of copies of $C$.
\end{itemize}
Classes of $R^{\rm op}$-modules of the form 
\[
{}^{\perp}C := \{ M \in \modR \, \text{:} \, \Ext^i_R(M,C) = 0 \, \text{for every} \, i > 0 \}
\]
are called \emph{$n$-cotilting classes} if $C$ is an $n$-cotilting $R^{\rm op}$-module.

\begin{theorem}\label{theo1}
The following assertions are equivalent for any ring $R$ and $0 < n < \infty$:
\begin{enumerate}[(a)]
\item $R$ is Lee $n$-coherent;

\item $(\mathsf{F}_n,\mathsf{A}_n)$ is a complete duality pair;

\item the natural morphism $\varinjlim_{\lambda \in \Lambda} \Ext^1_R(F,M_\lambda) \to \Ext^1_R(F,\varinjlim_{\lambda \in \Lambda} M_\lambda)$ is an isomorphism for any $F \in \mathsf{FP}^{\leq n}$ and any direct system $(M_\lambda)_{\lambda \in \Lambda}$ of $R$-modules;

\item $\mathsf{A}_n$ is closed under direct limits;

\item the natural morphism $\Tor_1^R(\prod_{\lambda \in \Lambda} N_\lambda, F) \to \prod_{\lambda \in \Lambda} \Tor_1^R(N_\lambda,F)$ is an isomorphism for any $F \in \mathsf{FP}^{\leq n}$ and any family $(N_\lambda)_{\lambda \in \Lambda}$ of $R^{\rm op}$-modules;

\item $\mathsf{F}_n$ is closed under direct products;

\item any direct product of copies of $R$ belongs to $\mathsf{F}_n$ as an $R^{\rm op}$-module;

\item $\mathsf{F}_n$ is preenveloping; 

\item $\mathsf{FP}^{\leq n} \subseteq \mathsf{FP}_2$;

\item $\mathsf{A}_n$ is coresolving;

\item $({}^{\perp_1}\mathsf{A}_n,\mathsf{A}_n)$ is a hereditary cotorsion pair;

\item $\mathsf{A}_n$ is an $n$-tilting class;

\item $\mathsf{F}_n$ is an $n$-cotilting class;

\item every $R$-module has an $n$-absolutely pure cover. \\
\end{enumerate}
\end{theorem}

\begin{proof} ~\
\begin{itemize}
\item (a) $\Rightarrow$ (b): By Proposition \ref{prop:basic_properties}, $(\mathsf{F}_n,\mathsf{A}_n)$ is a perfect duality pair. Also, $\mathsf{F}_n$ is closed under direct summands and finite coproducts. On the other hand, by Remark \ref{rmk:FP2}, we have that $\mathsf{FP}^{\leq n} \subseteq \mathsf{FP}_2$, and so by \cite[Lem. 1.2.11 (d)]{GT} there is a natural isomorphism $\Tor^R_1(M^+,F) \cong [\Ext^1_R(F,M)]^+$ for every $F \in \mathsf{FP}^{\leq n}$ and $M \in {\rm Mod}(R)$. Since $\mathbb{Q/Z}$ is an injective cogenerator in ${\rm Mod}(\mathbb{Z})$, it follows that $M \in \mathsf{A}_n$ if, and only if, $M^+ \in \mathsf{F}_n$.

\item (b) $\Rightarrow$ (a): Suppose we are given a finitely generated submodule $K$ of a finitely generated and free $R$-module $R^m$, such that $\pd(K) \leq n-1$. In order to see that $K$ is finitely presented, since it is already finitely generated, it suffices to show that $\Hom_R(K,\sim)$ preserves direct limits of FP-injective $R$-modules (see \cite[\S 25.4]{Wisbauer}). So suppose we are given a direct system $(M_\lambda)_{\lambda \in \Lambda}$ where each $M_\lambda$ is an FP-injective $R$-module. Consider the canonical pure epimorphism
\[
\bigoplus_{\lambda \in \Lambda} M_\lambda \to \varinjlim_{\lambda \in \Lambda} M_\lambda \to 0.
\]
Since FP-injective $R$-modules are $n$-absolutely pure, and $\mathsf{A}_n$ is closed under coproducts by Proposition \ref{prop:basic_properties} (5), we have that $\oplus_{\lambda \in \Lambda} M_\lambda \in \mathsf{A}_n$. On the other hand, by (b) we have that $(\mathsf{A}_n,\mathsf{F}_n)$ is a duality pair, and so $\mathsf{A}_n$ is closed under pure quotients. Thus, $\varinjlim_{\lambda \in \Lambda} M_\lambda \in \mathsf{A}_n$, and hence 
\[
\Ext^1_R(R^m / K, \varinjlim_{\lambda \in \Lambda} M_\lambda) = 0 \ \ \text{and} \ \ \Ext^1_R(R^m / K, M_\lambda) = 0
\] 
for every $\lambda \in \Lambda$. It follows that there exists the following commutative diagram with exact rows, where the vertical arrows are the canonical morphisms induced by the universal property of direct limits:
\[
\begin{tikzpicture}[description/.style={fill=white,inner sep=2pt}] 
\matrix (m) [matrix of math nodes, row sep=2em, column sep=1em, text height=1.25ex, text depth=0.25ex] 
{ 
\varinjlim_{\lambda \in \Lambda} \Hom_R(R^m/K, M_\lambda) & \varinjlim_{\lambda \in \Lambda} \Hom_R(R^m, M_\lambda) & \varinjlim_{\lambda \in \Lambda} \Hom_R(K, M_\lambda) & 0 \\
\Hom_R(R^m/K, \varinjlim_{\lambda \in \Lambda} M_\lambda) & \Hom_R(R^m, \varinjlim_{\lambda \in \Lambda} M_\lambda) & \Hom_R(K, \varinjlim_{\lambda \in \Lambda} M_\lambda) & 0 \\
}; 
\path[->] 
(m-1-1) edge (m-1-2) edge node[left] {\footnotesize$\alpha$} (m-2-1)
(m-1-2) edge (m-1-3) edge node[left] {\footnotesize$\beta$} (m-2-2)
(m-1-3) edge (m-1-4) edge node[left] {\footnotesize$\gamma$} (m-2-3)
(m-2-1) edge (m-2-2) (m-2-2) edge (m-2-3) (m-2-3) edge (m-2-4)
;
\end{tikzpicture} 
\]
Since $R^m$ and $R^m / K$ are finitely presented, we have that $\alpha$ and $\beta$ are isomorphisms. Then, the 5 Lemma implies that $\gamma$ is also an isomorphism. 

\item (a) $\Rightarrow$ (c) and (a) $\Rightarrow$ (e): Follows by Remark \ref{rmk:FP2} and \cite[Thm. 2]{Brown}.

\item (c) $\Rightarrow$ (d) and (e) $\Rightarrow$ (f) $\Rightarrow$ (g): Immediate. 

\item (d) $\Rightarrow$ (a): Similar to (b) $\Rightarrow$ (a).

\item (f) $\Leftrightarrow$ (h): Since $\mathsf{F}_n$ is closed under direct summands and pure submodules, we obtain from \cite[Coroll. 3.5 (c)]{RadaSaorin98} that $\mathsf{F}_n$ is preenveloping if, and only if, $\mathsf{F}_n$ is closed under direct products.  

\item (g) $\Rightarrow$ (a): Let $K \subseteq R^m$ be a finitely generated submodule of a finitely generated and free $R$-module $R^m$ with ${\rm pd}(K) \leq n-1$. By \cite[\S 25.4]{Wisbauer}, $K$ is finitely presented if, and only if, the canonical morphism $\varphi \colon R^{\Lambda} \otimes_R K \to K^\Lambda$ is an isomorphism for every set $\Lambda$. By (g), we have that $R^\Lambda \in \mathsf{F}_n$, and so $\Tor^R_1(R^\Lambda,R^m / K) = 0$. Thus, we can obtain the following commutative diagram with exact rows:
\[
\begin{tikzpicture}[description/.style={fill=white,inner sep=2pt}] 
\matrix (m) [matrix of math nodes, row sep=2em, column sep=2em, text height=1.25ex, text depth=0.25ex] 
{ 
0 & R^\Lambda \otimes_R K & R^\Lambda \otimes_R R^m & R^\Lambda \otimes_R R^m/K & 0 \\ 0 & K^\Lambda & (R^m)^\Lambda & (R^m/K)^\Lambda & 0 \\
}; 
\path[->] 
(m-1-1) edge (m-1-2)
(m-1-2) edge (m-1-3) edge node[left] {\footnotesize$\varphi$} (m-2-2)
(m-1-3) edge (m-1-4) edge node[left] {\footnotesize$\cong$} (m-2-3)
(m-2-1) edge (m-2-2) (m-2-2) edge (m-2-3) (m-2-3) edge (m-2-4) (m-2-4) edge (m-2-5)
(m-1-4) edge node[left] {\footnotesize$\cong$} (m-2-4) edge (m-1-5)
;
\end{tikzpicture} 
\]
where the middle and right vertical arrows are isomorphisms by \cite[\S 25.4]{Wisbauer}, since $R^m$ and $R^m / K$ are finitely presented. Thus, Five Lemma implies that $\varphi$ is an isomorphism. 

\item (a) $\Leftrightarrow$ (i): See Remark \ref{rmk:FP2}.

\item (i) $\Rightarrow$ (j): By Proposition \ref{prop:basic_properties} (4), it is only left to show that $\mathsf{A}_n$ is closed under monocokernels. Thus, consider a short exact sequence of the form $0 \to M_1 \to M_2 \to M_3 \to 0$ where $M_1, M_2 \in \mathsf{A}_n$, and let $F \in \mathsf{FP}^{\leq n}$. We have the induced exact sequence of abelian groups $\Ext^1_R(F,M_2) \to \Ext^1_R(F,M_3) \to \Ext^2_R(F,M_1)$ where $\Ext^1_R(F,M_2) = 0$. On the other hand, $\Ext^2_R(F,M_1) \cong \Ext^1_R(K,M_1)$, where $K \in \mathsf{FP}^{\leq n}$ since $F \in \mathsf{FP}_2$, and so $\Ext^1_R(K,M_1) = 0$. Hence, $\Ext^1_R(F,M_3) = 0$ for every $F \in \mathsf{FP}^{\leq n}$, that is, $M_3 \in \mathsf{A}_n$. 

\item (j) $\Leftrightarrow$ (k): Immediate by Proposition \ref{prop:basic_properties} (4).

\item (j) $\Leftrightarrow$ (l): We already know from Proposition \ref{prop:basic_properties} that $\mathsf{A}_n$ is preenveloping, closed under coproducts and direct summands, and that ${}^{\perp_1}[\mathsf{A}_n] \subseteq \mathsf{P}^{\leq n}$. The equivalence between (j) and (l) then follows by \cite[Thm. 5.1.14]{GT}. 

\item (m) $\Rightarrow$ (f): It is an immediate consequence of \cite[Thm. 8.1.9]{GT}.

\item (f) $\Rightarrow$ (m): We already know from Proposition \ref{prop:basic_properties} that $\mathsf{F}_n$ is covering, closed under direct summands, and $[\mathsf{F}_n]^{\perp_1} \subseteq \mathsf{I}^{\leq n}$. By (f), $\mathsf{F}_n$ is also closed under direct products. Furthermore, since (f) and (a) are equivalent, we have that $R$ is Lee $n$-coherent, which in turn implies that $\mathsf{F}_n$ is resolving (following an argument similar to (i) $\Rightarrow$ (j), using the long exact sequence of Tor). Thus, again by \cite[Thm. 8.1.9]{GT}, we have that $\mathsf{F}_n$ is $n$-cotilting. 

\item (b) $\Rightarrow$ (n): If we assume (b), then $(\mathsf{A}_n,\mathsf{F}_n)$ is a coproduct closed duality pair by Proposition \ref{prop:basic_properties} (5). Hence, $\mathsf{A}_n$ is covering.

\item (n) $\Rightarrow$ (d): We know that the class $\mathsf{A}_n$ is closed under pure submodules, extensions and direct products. If we assume (n), we have that $\mathsf{A}_n$ is also covering, and so it is closed under direct limits by \cite[Thm. 3.4]{DaiDing18}.
\end{itemize}
\end{proof}

\begin{remark}\label{rmk:Thm31YangLiu11} ~\
\begin{enumerate}[(1)]
\item The equivalence between (a) and (b) in the previous proposition can be immediately deduced from \cite[Thm. 3]{Lee}, and also the equivalence between (a), (f) and (g) is stated and proved in \cite[Thm. 5]{Lee}. However, these results depend  on \cite[Thm. 2]{Lee}, which has a gap in the proof of (ii') $\Rightarrow$ (i). Namely, the short exact sequence $0 \to A \to E \to E/A \to 0$ considered by the author, with $E$ injective and $A$ $n$-absolutely pure, need not be pure in general. For this reason, we have offered a different argument for the equivalence between (a), (b), (f) and (g).  

\item The equivalence between (a), (c)\footnote{We have corrected the direction of the natural morphism $\varinjlim_{\lambda \in \Lambda} \Ext^1_R(F,M_\lambda) \to \Ext^1_R(F,\varinjlim_{\lambda \in \Lambda} M_\lambda)$ induced by the universal property of colimits.}, (d), (h), (j) and (k) is stated and proved in \cite[Thm. 3.1]{YangLiu11}. Parts of our proof for Theorem \ref{theo1} simplifies some arguments in \cite[Thm. 3.1]{YangLiu11}. On the other hand, we fill a gap in the proof of the implication (3) $\Rightarrow$ (1) provided in \cite[Thm. 3.1]{YangLiu11}. In this reference, the authors consider a finitely generated submodule $K$ of a free $R$-module $F$ whose projective dimension is $n-1$, and claim that there is an isomorphism $\gamma \colon \Hom_R(K, \varinjlim M_i) \to \varinjlim \Hom_R(K,M_i)$, where $(M_i)_{i \in I}$ is a direct system in $\Rmod$. Certainly, this implies that $K$ is finitely presented by \cite[Prop. 2.2]{Stenstrom70}, and so $R$ is Lee $n$-coherent. However, it is not clear from the assertion (3) and the arguments given in \cite[Thm. 3.1]{YangLiu11} why $\gamma$ is an isomorphism. This partially has to do with the fact that it is not specified in which class of $\Rmod$ the modules $M_i$ are taken. 

\item The assertions in the previous theorem are not necessarily equivalent to the following: 
\begin{enumerate}[(o)]
\item[(o)] $\mathsf{F}_n$ is a resolving class of ${\rm Mod}(R^{\rm op})$.

\item[(p)] $(\mathsf{F}_n,[\mathsf{F}_n]^{\perp_1})$ is a hereditary cotorsion pair. 
\end{enumerate}
These two assertions are clearly equivalent, and assertion (a) certainly implies (o). However, (o) does not imply (a) in general. Take for instance $n = \infty$. Then, $\mathsf{F}_n$ is the class of flat $R^{\rm op}$-modules, which is resolving for arbitrary (and so not necessarily coherent = $\infty$-coherent) rings $R$. 

\item Pinzon proves in \cite[Coroll. 2.7]{Pinzon08} that if R is coherent (i.e., $\infty$-coherent) then every $R$-module has an absolutely pure cover (i.e., $\infty$-absolutely pure). The converse of this result is proved in \cite[Coroll. 3.5]{DaiDing18} by Dai and Ding, solving in the positive an open question posed by Pinzon herself in \cite[Rmk. 2.8]{Pinzon08}. The equivalence between (a) and (n) in the previous theorem represents a generalization of this characterization of coherent rings. 
\end{enumerate}
\end{remark}

%\begin{remark}
%Moreover, $(\mathsf{F}_n,\mathsf{A}_n)$ and $(\mathsf{BF}_n,\mathsf{BA}_n)$ are bicomplete duality pairs if, and only if, $R$ is Lee $n$-coherent. {\color{red}{Can be apply the theory of relative Gorenstein flats and semi-dualizing modules (last two papers with Vítor)?}}
%\end{remark}

More properties of $\mathsf{F}_n$ and $\mathsf{A}_n$ over Lee $n$-coherent rings are given in \cite[Props. 3.4 \& 3.7]{YangLiu11}, concerning the relation between the classes $\mathsf{F}_n$ and $\mathsf{A}_n$, their orthogonal complements, and the classes of projective, flat, cotorsion and FP-injective modules. First, the authors show in \cite[Prop. 3.4]{YangLiu11} that, if $R$ is a Lee $n$-coherent ring, then the following conditions are equivalent:
\begin{enumerate}[(a)]
\item Every $n$-flat $R^{\rm op}$-module is flat.

\item Every cotorsion $R^{\rm op}$-module belongs to $[\mathsf{F}_n]^{\perp_1}$.

\item Every $n$-absolutely pure $R$-module is FP-injective.

\item Every finitely presented $R$-module belongs to ${}^{\perp_1}[\mathsf{A}_n]$\footnote{We just point out a typo in assertion (4) of \cite[Prop. 3.4]{YangLiu11}, it should say that every finitely presented $R$-module belongs to ${}^{\perp_1}[\mathsf{A}_n]$ (instead of $[\mathsf{A}_n]^{\perp_1}$, as originally stated).}.
\end{enumerate}

Over an arbitrary ring $R$, it is shown in \cite[Prop. 3.7]{YangLiu11} that the following conditions are equivalent:
\begin{enumerate}[(a)]
\item Every $R$-module is $n$-absolutely pure.

\item Every cotorsion $R$-module is $n$-absolutely pure.

\item Every $R^{\rm op}$-module is $n$-flat.

\item Every cotorsion $R^{\rm op}$-module is $n$-flat.

\item Every $R^{\rm op}$-module in $[\mathsf{F}_n]^{\perp_1}$ is injective.

\item Every $R$-module in ${}^{\perp_1}[\mathsf{A}_n]$ is projective.

\item Every nonzero $R^{\rm op}$-module contains a nonzero $n$-flat submodule.

\item The complete cotorsion pair $({}^{\perp_1}[\mathsf{A}_n],\mathsf{A}_n)$ is hereditary and every $R$-module in ${}^{\perp_1}[\mathsf{A}_n]$ is $n$-absolutely pure. 
\end{enumerate}

%%%%%%%%%%%%%%%%%%%%%%%%%%%%%%%%%%%%%
%%%%%%%%%%%%%%%%%%%%%%%%%%%%%%%%%%%%%
%%%%%%%%%%%%%%%%%%%%%%%%%%%%%%%%%%%%%
%%%%%%%%%%%%%%%%%%%%%%%%%%%%%%%%%%%%%

\section{$n$-Coherency at the level of ideals}\label{sec:ideals}

It is well known that coherent rings can be characterized by means of finiteness conditions on their ideals. Recall that a ring $R$ is left coherent if every finitely generated left ideal of $R$ is finitely presented. Then, the following question is natural:

\begin{question}\label{q:open1}
Given a ring $R$, is $R$ Lee $n$-coherent provided that every finitely generated left ideal with ${\rm pd}(I) \leq n-1$ is finitely presented?
\end{question}

In other words, are Lee $n$-coherent rings and weakly Lee $n$-coherent rings the same class of rings? We are not aware of an answer to this question. However, weakly Lee $n$-coherent rings will be of interest by themselves, and will be related to the classes $\mathsf{BA}_n$ and $\mathsf{BF}_n$ defined at the end of Section \ref{sec:general_properties}, as stated in the following result (whose proof is similar to that of Theorem \ref{theo1}).

\begin{theorem}\label{theo2}
The following assertions are equivalent:
\begin{enumerate}[(a)]
\item $R$ is weakly Lee $n$-coherent;

\item $(\mathsf{BF}_n,\mathsf{BA}_n)$ is a complete duality pair;

\item $\mathsf{BA}_n$ is closed under direct limits;

\item the natural morphism $\varinjlim_{\lambda \in \Lambda} \Ext^1_R(R/I,M_\lambda) \to \Ext^1_R(R/I,\varinjlim_{\lambda \in \Lambda} M_\lambda)$ is an isomorphism for any finitely generated left ideal $I \lhd R$ with ${\rm pd}(I) \leq n-1$, and any direct system $(M_\lambda)_{\lambda \in \Lambda}$ of $R$-modules;

\item $\mathsf{BF}_n$ is preenveloping; 

\item $\mathsf{BF}_n$ is closed under direct products;

\item any direct product of copies of $R$ belongs to $\mathsf{BF}_n$ as an $R^{\rm op}$-module;

\item the natural morphism $\Tor_1^R(\prod_{\lambda \in \Lambda} N_\lambda, R/I) \to \prod_{\lambda \in \Lambda} \Tor_1^R(N_\lambda,R/I)$ is an isomorphism for any finitely generated left ideal $I \lhd R$ with ${\rm pd}(I) \leq n-1$, and any family $(N_\lambda)_{\lambda \in \Lambda}$ of $R^{\rm op}$-modules;

\item every $R$-module has a Baer $n$-absolutely pure cover.
\end{enumerate}
\end{theorem}

\begin{remark}\label{rmk:theo2}
If we compare the statements of Theorems \ref{theo1} and \ref{theo2}, we can note in the second one the absence of the following five assertions: 
\begin{enumerate}[(i)]
\item $\mathsf{CFP}^{\leq n} \subseteq \mathsf{FP}_\infty$.

\item $\mathsf{BA}_n$ is coresolving;

\item the complete cotorsion pair $({}^{\perp_1}[\mathsf{BA}_n],\mathsf{BA}_n)$ is hereditary;

\item $\mathsf{BA}_n$ is an $n$-tilting class (with $n < \infty$);

\item $\mathsf{BF}_n$ is an $n$-cotilting class (with $n < \infty$). 
\end{enumerate}
Assertions, (ii) and (iii) are clearly equivalent, while (ii) and (iv) are equivalent by \cite[Thm. 5.1.14]{GT} and Proposition \ref{prop:basic_properties_cyclic}. The equivalence between (ii) and (v) can be reduced to showing that $\mathsf{BA}_n$ is closed under monocokernels if, and only if, $\mathsf{BF}_n$ is closed under taking epikernels and direct products (see Proposition \ref{prop:basic_properties_cyclic}). If we assume the former, then $\mathsf{BF}_n$ is clearly closed under epikernels, since $(\mathsf{BF}_n,\mathsf{BA}_n)$ is a duality pair. We can also show that $R$ is weakly Lee $n$-coherent (which in turn will imply that $\mathsf{BF}_n$ is closed under direct products by Theorem \ref{theo2}). Indeed, let $I \lhd R$ be a finitely generated left ideal with ${\rm pd}(I) \leq n-1$, and consider an FP-injective $R$-module $M$ (and so $M \in \mathsf{BA}_n$), then 
\[
\Ext^1_R(I,M) \cong \Ext^2_R(R/I, M) \cong \Ext^1_R(R/I,\Omega^{-1}(M)) = 0,
\]
where $\Omega^{-1}(M)$ denotes any injective cosyzygy of $M$, and so belongs to $\mathsf{BA}_n$ by (ii). Then, by \cite[Thm. 2.1.10]{Glaz89} we obtain that $I$ is finitely presented. Now if we assume that $\mathsf{BF}_n$ is closed under epikernels and direct products, we have in particular that $R$ is weakly Lee $n$-coherent, and so $(\mathsf{BA}_n, \mathsf{BF}_n)$ is a duality pair by Theorem \ref{theo2}. The latter implies that $\mathsf{BA}_n$ is closed under monocokernels if $\mathsf{BF}_n$ is closed under epikernels. 

From the previous, any of the last four conditions implies in particular that $R$ is weakly Lee $n$-coherent. The latter can also be deduced from (i). For, if $I \lhd R$ is a finitely generated left ideal with ${\rm pd}(I) \leq n-1$, then $R/I \in \mathsf{CFP}^{\leq n} \subseteq \mathsf{FP}_\infty$, and so $I \in \mathsf{FP}_\infty$ by \cite[Prop. 2.3]{BGH14}. In particular, $I$ is finitely presented. \\
\end{remark}

\begin{question}\label{q:weaklyn-coh} ~\
\begin{enumerate}[(1)]
\item Does any of the five assertions from the previous remark holds if $R$ is weakly Lee $n$-coherent? 

\item Are these assertions equivalent under the latter assumption? 
\end{enumerate}
\end{question}

We now study the relation between $\mathsf{A}_n$ and $\mathsf{BA}_n$, and between $\mathsf{F}_n$ and $\mathsf{BF}_n$. \\

\begin{proposition}\label{prop:An_vs_IAn}
Consider the following assertions:
\begin{enumerate}[(a)] 
\item $\mathsf{A}_n = \mathsf{BA}_n$. 

\item $\mathsf{F}_n = \mathsf{BF}_n$. 
\end{enumerate}
The implication (a) $\Rightarrow$ (b) always holds, and both assertions are equivalent if $R$ is weakly Lee $n$-coherent.
\end{proposition}

\begin{proof}
Suppose that $\mathsf{A}_n = \mathsf{BA}_n$, and let $N \in \mathsf{BF}_n$. Since $(\mathsf{BF}_n,\mathsf{BA}_n)$ is a duality pair, we have that $N^+ \in \mathsf{BA}_n = \mathsf{A}_n$, and so $N \in \mathsf{F}_n$ since $(\mathsf{F}_n,\mathsf{A}_n)$ is also a duality pair. 

Now assume that $\mathsf{F}_n = \mathsf{BF}_n$ holds and that $R$ is weakly Lee $n$-coherent. Let $M \in \mathsf{BA}_n$. By Theorem \ref{theo2}, $(\mathsf{BA}_n,\mathsf{BF}_n)$ is a duality pair, and so $M^+ \in \mathsf{BF}_n = \mathsf{F}_n$. On the other hand, since $(\mathsf{F}_n,\mathsf{A}_n)$ is a duality pair, we have that $M^{++} \in \mathsf{A}_n$. Thus, we have a pure embedding $0 \to M \to M^{++}$ where $M^{++} \in \mathsf{A}_n$, and since $\mathsf{A}_n$ is closed under pure submodules, we obtain $M \in \mathsf{A}_n$. Hence, $\mathsf{BA}_n = \mathsf{A}_n$. \\
\end{proof}

The previous result, along with Theorems \ref{theo1} and \ref{theo2}, gives a sufficient condition under which every weakly Lee $n$-coherent ring is Lee $n$-coherent. \\

\begin{corollary}\label{cor:when_weak_is_strong}
Let $R$ be a weakly Lee $n$-coherent ring. If either $\mathsf{A}_n = \mathsf{BA}_n$ or $\mathsf{F}_n = \mathsf{BF}_n$, then $R$ is Lee $n$-coherent. \\
\end{corollary}

\begin{remark}\label{rmk0} ~\
\begin{enumerate}[(1)]
\item In \cite[Lem. 2.1]{YangLiu11}, the authors claim that assertions (a) and (b) in Proposition \ref{prop:An_vs_IAn} always hold (so in particular, they would be independent). However, there is a gap in their proof for assertion (a). Concretely, the morphisms $R \to L$ and $A \to R^{n-1}$ in the pullback construction appearing in their proof need not be epic in general.

\item Let us consider the case $n = \infty$. Then $\mathsf{A}_\infty$ is the class of FP-injective $R$-modules, and assertion (1) becomes: $M \in R\textrm{-}\mathrm{Mod}$ is FP-injective if, and only if, $\Ext^1_R(R/I,M) = 0$ for any finitely generated ideal $I$. This equivalence was proved to be true under the assumption that $R$ is coherent (see for instance Megibben's \cite[Thm. 4 \& the remark afterwards]{Megibben70} or Stenström's \cite[Lem. 3.1]{Stenstrom70}). We are not aware if it is possible to drop this assumption. Actually, this is an open problem posed in \cite{Damiano79,Shen11}. 

Let us provide more context to this open question. In \cite{Gupta69}, the author introduces the notion of $f$-injective modules. An $R$-module $M$ is said to be \textbf{$\bm{f}$-injective} if every homomorphism $f \colon I \to M$ can be extended to a homomorphism of $R$ into $M$, with $I$ running over the set of all finitely generated left ideals $I \lhd R$. These modules are also called \textbf{coflat} in \cite[Def. 1.5 \& Prop. 1.6]{Damiano79}. In \cite{Shen11}, the author asserts that it is unknown whether a right $f$-injective ring $R$ is right FP-injective or not. \\
\end{enumerate}
\end{remark}

The possibility to characterize Lee $n$-coherent rings in terms of ideals has been considered before in \cite{YangLiu11}. Specifically, in \cite[Prop. 3.1]{YangLiu11} the authors claim that the following assertions are equivalent:
\begin{enumerate}[(a)]
\item $R$ is Lee $n$-coherent;

\item $\Ext^1_R(I,N) = 0$ for every FP-injective $R$-module $N$ and every finitely generated ideal $I$ of $R$ with ${\rm pd}(I) \leq n-1$;

\item $\Ext^2_R(R/I,N) = 0$ for every FP-injective $R$-module $N$ and every finitely generated ideal $I$ of $R$ with ${\rm pd}(I) \leq n-1$;
 
\item if $0 \to N \to M \to L \to 0$ is an exact sequence of $R$-modules where $N$ is FP-injective and $M \in \mathsf{A}_n$, then $L \in \mathsf{A}_n$. \\
\end{enumerate}

\begin{remark}\label{obs-ideal} ~\
\begin{enumerate}[(1)]
\item One can note that the implications (a) $\Rightarrow$ (b) $\Leftrightarrow$ (c) are immediate. Regarding the proof of (d) $\Rightarrow$ (a) provided by the authors in \cite[Prop. 3.1]{YangLiu11}, there is a gap in their conclusion. Specifically, from (d), they show that $\Ext^1_R(I, N) \cong \Ext^2_R(R/I, N) = 0$ for every finitely generated left ideal of $R$ with $\textrm{pd}(I) \leq n - 1$ and every FP-injective $R$-module $N$, and thus the implication (d) $\Rightarrow$ (b) holds. Moreover, by \cite[Thm. 2.1.10]{Glaz89}, we obtain that $I$ is finitely presented. At this point, the conclusion in \cite[proof of (4) $\Rightarrow$ (1) in Prop. 3.1]{YangLiu11} is that $R$ is Lee $n$-coherent. However, the condition that defines such rings is only verified at the level of ideals (that is, that $R$ is weakly Lee $n$-coherent), and we are not aware if this alone suffices to conclude $n$-coherency in the sense of Lee (as it does for coherency).

Let us fix this gap and show that (a) and (d) are equivalent. The implication (a) $\Rightarrow$ (d) is clear by Theorem \ref{theo1}. Now let us assume (d), and consider a direct system $(M_i)_{i \in I}$ of FP-injective $R$-modules. We have the canonical pure exact sequence 
\[
0 \to A \to \bigoplus_{i \in I} M_i \to \varinjlim_{i \in I} M_i \to 0
\]
where $\oplus_{i \in I} M_i$ is FP-injective by \cite[Coroll. 2.4]{Stenstrom70}. Since FP-injectives are closed under pure submodules, we have that $A$ is FP-injective. On the other hand, $\oplus_{i \in I} M_i \in \mathsf{A}_n$. Thus, by (d) we obtain that $\varinjlim_{i \in I} M_i \in \mathsf{A}_n$. The rest of the proof follows as in the proof of (d) $\Rightarrow$ (a) in Theorem \ref{theo1}. 

So far we have (d) $\Leftrightarrow$ (a) $\Rightarrow$ (b) $\Leftrightarrow$ (c). A proof for the implication (c) $\Rightarrow$ (d) is given in \cite{YangLiu11}, but the authors use \cite[Lem. 2.1]{YangLiu11} (see Remark \ref{rmk0}). In conclusion, we can only guarantee the implication (a) $\Rightarrow$ (b) and the equivalence between (a) and (d), and between (b) and (c), but not between all of these assertions. Note that assertion (d) can be part of the list of equivalent assertions in Theorem \ref{theo1}. 

\item In \cite[Rmk. 2.2]{ElHaddaouiKimMahdou26}, Question \ref{q:weaklyn-coh} is answered in the positive. However, we point out that the argument given there depends on \cite[Prop. 3.1]{YangLiu11}. 
\end{enumerate}
\end{remark}

There is another claim in \cite{YangLiu11} regarding the characterization of Lee $n$-coherency in terms of ideals. This is the case of \cite[Prop. 3.3]{YangLiu11}. It is stated that the following assertions are equivalent: 
\begin{enumerate}[(a)]
\item The complete cotorsion pair $({}^{\perp_1}\mathsf{A}_n,\mathsf{A}_n)$ is hereditary;

\item $R$ is Lee $n$-coherent and the cotorsion pair $(\mathsf{F}_n,[\mathsf{F}_n]^{\perp_1})$ is hereditary;

\item $\Ext^2_R(R/I,M) = 0$ for every $R / I \in \mathsf{CFP}^{\leq n}$ and $M \in \mathsf{A}_n$;

\item $R$ is Lee $n$-coherent and $\Tor^R_2(N,R/I) = 0$ for every $R / I \in \mathsf{CFP}^{\leq n}$ and $N \in \mathsf{F}_n$. 
\end{enumerate}

\begin{remark}
There is a gap similar to the one pointed out in \cite[Prop. 3.1]{YangLiu11}. Note that in assertion (b), $(\mathsf{F}_n,[\mathsf{F}_n]^{\perp_1})$ is automatically a hereditary cotorsion pair since $R$ is Lee $n$-coherent (see Remark \ref{rmk:Thm31YangLiu11} (3)). In particular, $\mathsf{F}_n$ is resolving in this case, and so the Tor-orthogonality mentioned in (d) is redundant. It follows that (b) and (d) are the same assertion, namely ``$R$ is Lee $n$-coherent''. It then follows that (a), (b) and (d) are automatically equivalent by Theorem \ref{theo1}. Also, (a) $\Rightarrow$ (c) follows by a dimension shifting argument, since (a) implies that any injective $1$-cosyzygy of an $n$-absolutely pure $R$-module is $n$-absolutely pure. 

Finally, regarding (c) $\Rightarrow$ (b), in \cite[Prop. 3.3]{YangLiu11} it is claimed that this implication follows by \cite[Thm. 1]{Lee}. However, we think there are missing details here. On the one hand, \cite[Thm. 1]{Lee} asserts that $R$ is (Lee) $n$-coherent if, and only if, $\Ext^2_R(F,M) = 0$ for every $F \in \mathsf{FP}^{\leq n}$ and $M \in \mathsf{A}_n$. On the other hand, it is not known whether the latter is equivalent to (c), by the same reason we claim that the assertions in \cite[Lem. 2.1]{YangLiu11} are open (see Remark \ref{rmk0}). \\
\end{remark}

We close this section stating and proving the analogs of \cite[Props. 3.4 \& 3.7]{YangLiu11} for $\mathsf{BA}_n$ and $\mathsf{BF}_n$.

\begin{proposition}
Consider the following assertions:
\begin{enumerate}[(a)]
\item Every $R^{\rm op}$-module in $\mathsf{BF}_n$ is flat.

\item Every cotorsion $R^{\rm op}$-module belongs to $[\mathsf{BF}_n]^{\perp_1}$.

\item Every $R$-module in $\mathsf{BA}_n$ is FP-injective.

\item Every finitely presented $R$-module belongs to ${}^{\perp_1}[\mathsf{BA}_n]$.
\end{enumerate}
The implications (a) $\Leftrightarrow$ (b) $\Leftarrow$ (c) $\Leftrightarrow$ (d) hold. If in addition, $R$ is weakly Lee $n$-coherent, then (a) $\Rightarrow$ (c), and so the four assertions are equivalent. 
\end{proposition}

\begin{proof}
Consider the class $\mathsf{F}$ of flat $R^{\rm op}$-modules. Since $(\mathsf{F},\mathsf{F}^{\perp_1})$ and $(\mathsf{BF}_n,[\mathsf{BF}_n]^{\perp_1})$ are cotorsion pairs, the equivalence between (a) and (b) follows by \cite[comments right after Def. 7.1.1]{EJ00}. The equivalence between (c) and (d) follows in a similar way. Moreover, if we let $N \in \mathsf{BF}_n$, then $N^+ \in \mathsf{BA}_n$ by Proposition \ref{prop:basic_properties_cyclic}. If we assume (c), we have that $N^+$ is FP-injective. It follows by \cite[Thm. 2.1]{Fieldhouse72} that $N$ is flat. Hence, (c) $\Rightarrow$ (a).

Now assume that every $R^{\rm op}$-module in $\mathsf{BF}_n$ is flat, under the additional hypothesis that $R$ is weakly Lee $n$-coherent. So for every $M \in \mathsf{BA}_n$, we have that $M^+ \in \mathsf{BF}_n$ by Theorem \ref{theo2}, and so $M^+$ is flat. Again by \cite[Thm. 2.1]{Fieldhouse72}, we have that $M^{++}$ is FP-injective. Since FP-injectives are closed under pure submodules, and $M$ is pure in $M^{++}$ (see \cite[Prop. 5.3.9]{EJ00}), we have that $M$ is FP-injective. 
\end{proof}

\begin{proposition}
The following assertions are equivalent:
\begin{enumerate}[(a)]
\item Every $R$-module belongs to $\mathsf{BA}_n$.

\item Every cotorsion $R$-module belongs to $\mathsf{BA}_n$.

\item Every $R^{\rm op}$-module belongs to $\mathsf{BF}_n$.

\item Every cotorsion $R^{\rm op}$-module belongs to $\mathsf{BF}_n$.

\item Every $R^{\rm op}$-module in $[\mathsf{BF}_n]^{\perp_1}$ is injective.

\item Every $R$-module in ${}^{\perp_1}[\mathsf{BA}_n]$ is projective.

\item Every nonzero $R^{\rm op}$-module contains a nonzero submodule in $\mathsf{BF}_n$.

\item The complete cotorsion pair $({}^{\perp_1}[\mathsf{BA}_n],\mathsf{BA}_n)$ is hereditary and every $R$-module in ${}^{\perp_1}[\mathsf{BA}_n]$ belongs to $\mathsf{BA}_n$. 
\end{enumerate}
\end{proposition}

\begin{proof}
If we assume (a) and let $N \in \modR$, then $N^+ \in \mathsf{BA}_n$, and so $N \in \mathsf{BF}_n$ since $(\mathsf{BF}_n,\mathsf{BA}_n)$ is a duality pair. Now if we assume (c) and let $M \in \Rmod$, then $M^+ \in \mathsf{BF}_n$, and so  $M^{++} \in \mathsf{BA}_n$. Since there is a pure embedding $M \to M^{++}$ and $\mathsf{BA}_n$ is closed under pure submodules, we conclude that $M \in \mathsf{BA}_n$. Hence, (a) and (c) are equivalent. In a similar way, one can show that (b) and (d) are equivalent. 

From the facts that $(\mathsf{BF}_n,[\mathsf{BF}_n]^{\perp_1})$ and $({}^{\perp_1}[\mathsf{BA}_n],\mathsf{BA}_n)$ are cotorsion pairs, we have that (c) $\Leftrightarrow$ (e) and (a) $\Leftrightarrow$ (f).

The implication (a) $\Rightarrow$ (b) is immediate. Now assume (b) and let $M \in \Rmod$. Since $(\mathsf{F},\mathsf{F}^{\perp_1})$ is a complete cotorsion pair, we can consider a short exact sequence $0 \to M \to C \to L \to 0$ where $C$ is cotorsion (and so $C \in \mathsf{BA}_n$) and $L$ is flat. This sequence is pure, and so $M \in \mathsf{BA}_n$. %The equivalence between (c) and (d) follows similarly. 

So far we have that the first six assertions are equivalent. The implication (a) $\Rightarrow$ (h) is immediate. Now if we let $M \in \Rmod$ and assume (h), we have a short exact sequence $0 \to A \to C \to M \to 0$ where $A \in \mathsf{BA}_n$ and $C \in {}^{\perp_1}[\mathsf{BA}_n] \subseteq \mathsf{BA}_n$. Since $({}^{\perp_1}[\mathsf{BA}_n],\mathsf{BA}_n)$ is hereditary, we have that $M \in \mathsf{BA}_n$. Hence, (a) and (h) are equivalent. 

Finally, (c) $\Rightarrow$ (g) is evident, while (g) $\Rightarrow$ (e) follows as \cite[proof of (7) $\Rightarrow$ (5) in Prop. 3.7]{YangLiu11}.
\end{proof}

%%%%%%%%%%%%%%%%%%%%%%%%%%%%%%%%%%%%%
%%%%%%%%%%%%%%%%%%%%%%%%%%%%%%%%%%%%%
%%%%%%%%%%%%%%%%%%%%%%%%%%%%%%%%%%%%%
%%%%%%%%%%%%%%%%%%%%%%%%%%%%%%%%%%%%%

\section{Approximations by $n$-absolutely pure and $n$-flat modules}\label{sec:approximations}

Besides the left and right approximations by the classes $\mathsf{A}_n$, $\mathsf{BA}_n$, $\mathsf{F}_n$ and $\mathsf{BF}_n$ found in the previous sections, it is possible to obtain some more under certain conditions. Some of these new approximations will involve the unique mapping property. 

Approximations with the unique mapping property have been previously considered in the literature. For instance, in \cite{AMMH93} the authors study flat envelopes with the unique mapping property. The relation between coherent rings and projective and flat envelopes with the unique mapping property is investigated in \cite{Ding96}. Given a class $\mathcal{A}$ of $R$-modules and an $\mathcal{A}$-cover $\rho \colon A \to M$, we say that $\rho$ has the \emph{unique mapping property (u.m.p.)} if for every morphism $\rho' \colon A' \to M$ with $A' \in \mathcal{A}$ there exists a unique $h \colon A' \to A$ such that $\rho' = \rho \circ h$. 

With the help of Theorem \ref{theo1} we can improve \cite[Prop. 3.2]{YangLiu11}, which asserts that if $R$ is a Lee $n$-coherent ring, then the following conditions are equivalent:
\begin{enumerate}[(a)]
\item every $R$-module has an $n$-absolutely pure cover with the u.m.p.;

\item for every exact sequence $A \to B \to C \to 0$ in $\Rmod$ with $A, B \in \mathsf{A}_n$, then $C \in \mathsf{A}_n$ (in other words, $\mathsf{A}_n$ is closed under cokernels).
\end{enumerate}

If we assume either (a) or (b), we have by Theorem \ref{theo1} that $R$ is Lee $n$-coherent. We can thus restate \cite[Prop. 3.2]{YangLiu11} as follows (specifically, we can remove the hypothesis that $R$ is Lee $n$-coherent).

\begin{proposition}\label{prop:ump}
The following assertions hold:
\begin{enumerate}[(1)]
\item The following are equivalent:
\begin{enumerate}[(a)]
\item every $R$-module has an $n$-absolutely pure cover with the u.m.p.;

\item $\mathsf{A}_n$ is closed under cokernels.
\end{enumerate}

\item The following are equivalent:
\begin{enumerate}[(a)]
\item every $R$-module has a Baer $n$-absolutely pure cover with the u.m.p.;

\item $\mathsf{BA}_n$ is closed under cokernels.
\end{enumerate}
\end{enumerate}
\end{proposition}

\begin{proof}
Both equivalences follow from the arguments given in the proof of \cite[Prop. 3.2]{YangLiu11} and Theorems \ref{theo1} and \ref{theo2}.
\end{proof}

There are in \cite{YangLiu11} two more results on left and right approximations by $n$-absolutely pure and $n$-flat modules. The first one characterizes when $n$-flat preenvelopes are monic. Specifically, in \cite[Prop. 3.5]{YangLiu11} it is proved that the following conditions are equivalent:
\begin{enumerate}[(a)]
\item Every $R^{\rm op}$-module has a monic $n$-flat preenvelope.

\item $R$ is Lee $n$-coherent and every flat $R$-module is $n$-absolutely pure. 

\item $R$ is Lee $n$-coherent and ${}_R R$ is $n$-absolutely pure. 

\item $R$ is Lee $n$-coherent and $(\mathsf{A}_n,[\mathsf{A}_n]^{\perp_1})$ is a perfect cotorsion pair.

\item $R$ is Lee $n$-coherent and every $R$-module has an epic $n$-absolutely pure cover.
\end{enumerate}

We can show that there is an analogous equivalence with respect to the classes $\mathsf{BA}_n$ and $\mathsf{BF}_n$.

\begin{proposition}
The following assertions are equivalent:
\begin{enumerate}[(a)]
\item Every $R^{\rm op}$-module has a monic Baer $n$-flat preenvelope.

\item $R$ is weakly Lee $n$-coherent and every flat $R$-module is Baer $n$-absolutely pure. 

\item $R$ is weakly Lee $n$-coherent and ${}_R R$ is Baer $n$-absolutely pure. 

\item $R$ is weakly Lee $n$-coherent and $(\mathsf{BA}_n,[\mathsf{BA}_n]^{\perp_1})$ is a perfect cotorsion pair.

\item $R$ is weakly Lee $n$-coherent and every $R$-module has an epic Baer $n$-absolutely pure cover.
\end{enumerate}
\end{proposition}

\begin{proof}
We can follow the arguments given in \cite[Prop. 3.5]{YangLiu11}:
\begin{itemize}
\item (a) $\Rightarrow$ (b): Assertion (a) implies that $\mathsf{BF}_n$ is preenveloping, and so $R$ is weakly Lee $n$-coherent and $(\mathsf{BA}_n,\mathsf{BF}_n)$ is a duality pair by Theorem \ref{theo2}. Now let $L$ be a flat $R$-module. Then, $L^+$ is injective by \cite[Thm. 3.2.10]{EJ00}. If we consider a monic $\mathsf{BF}_n$-preenvelope $L^+ \to N$ (with $N \in \mathsf{BF}_n$), it follows that $L^+$ is a direct summand of $N$, and so $L^+ \in \mathsf{BF}_n$. Now since $(\mathsf{BA}_n,\mathsf{BF}_n)$ is a duality pair, we get that $L \in \mathsf{BA}_n$. 

\item (b) $\Rightarrow$ (c): Immediate.

\item (c) $\Rightarrow$ (d): Since $R$ is weakly Lee $n$-coherent and ${}_R R \in \mathsf{BA}_n$, we have by Theorem \ref{theo2} and Proposition \ref{prop:basic_properties_cyclic} that $(\mathsf{BA}_n,\mathsf{BF}_n)$ is a perfect duality pair, and hence $(\mathsf{BA}_n,[\mathsf{BA}_n]^{\perp_1})$ is a perfect cotorsion pair. 

\item (d) $\Rightarrow$ (e): Immediate. 

\item (e) $\Rightarrow$ (a): Since $R$ is weakly Lee $n$-coherent, we have that $\mathsf{BF}_n$ is preenveloping and $(\mathsf{BA}_n,\mathsf{BF}_n)$ is a duality pair by Theorem \ref{theo2}. Let $M \in \modR$ and consider a $\mathsf{BF}_n$-preenvelope $\psi \colon M \to N$ (with $N \in \mathsf{BF}_n$). If $A \to M^+$ is an epic $\mathsf{BA}_n$-cover, we have a monomorphism $M^{++} \to A^+$ where $A^+ \in \mathsf{BF}_n$. After composing with the pure embedding $M \to M^{++}$, we get a monomorphism $f \colon M \to A^+$. Note that there is $g \colon N \to A^+$ such that $f = g \, \psi$, and since $f$ is monic, we have that $\psi$ is monic. 
\end{itemize}
\end{proof}

The following result extends \cite[Prop. 3.6 ]{YangLiu11}, where the authors study when modules have epic $n$-flat preenvelopes. \\

\begin{proposition}\label{prop:injectives_closed_under_quotients} 
The following assertions are equivalent:
\begin{enumerate}[(a)]
\item ${\rm pd}(F) \leq 1$ for every $F \in \mathsf{FP}^{\leq n}$.

\item $\mathsf{A}_n$ is closed under quotients.

\item For every $M \in \Rmod$ there is an exact sequence $$0 \to M \to A^0 \to A^1 \to 0$$ with $A^0, A^1 \in \mathsf{A}_n$.

\item $\mathsf{A}_n = \mathsf{A}_1$.

\item $\mathsf{A}_n$ is a $1$-tilting class.

\item Every $R$-module has a monic $n$-absolutely pure cover. 

\item $R$ is Lee $n$-coherent and ${\rm fd}(F) \leq 1$ for every $F \in \mathsf{FP}^{\leq n}$.

\item $R$ is Lee $n$-coherent and for every $N \in \modR$ there is an exact sequence $$0 \to F_1 \to F_0 \to N \to 0$$ with $F_0, F_1 \in \mathsf{F}_n$.

\item $R$ is Lee $n$-coherent and $\mathsf{F}_n$ is closed under submodules.

\item $R$ is Lee $n$-coherent and $\mathsf{F}_n = \mathsf{F}_1$.

\item $R$ is Lee $n$-coherent and $\mathsf{F}_n$ is a $1$-cotilting class.

\item Every $R^{\rm op}$-module has an epic $n$-flat preenvelope. 
\end{enumerate}
\end{proposition}

\mbox{}

\begin{proof} ~\
\begin{itemize}
\item (a) $\Leftrightarrow$ (b) $\Leftrightarrow$ (c): Follows by dimension shifting and the fact that $\mathsf{A}_n$ contains the class of injective $R$-modules. 

\item (a) $\Rightarrow$ (d): Immediate. 

\item (d) $\Rightarrow$ (b): Follows by noticing that $\mathsf{A}_1$ is always closed under quotients, since it is the right Ext-orthogonal complement of $\mathsf{FP}^{\leq 1}$ (see \cite[Prop. 2.11]{BOPP22}). 

\item (b) $\Leftrightarrow$ (e): One can note from the definition that 1-tilting classes are always closed under quotients. On the other hand, if $\mathsf{A}_n$ is closed under quotients, it follows by \cite[Thm. 6.1.4]{GT} and Proposition \ref{prop:basic_properties} that $\mathsf{A}_n$ is $1$-tilting. 

\item (b) $\Rightarrow$ (f): If $\mathsf{A}_n$ is closed under quotients, it is in particular closed under monocokernels, and so $R$ is Lee $n$-coherent. It follows by Theorem \ref{theo1} that $\mathsf{A}_n$ is covering. Now let $M \in \Rmod$ and $\varphi \colon A \to M$ be an $n$-absolutely pure cover of $M$. Since $\mathsf{A}_n$ is closed under quotients, we have that ${\rm Im}(\varphi) \in \mathsf{A}_n$. On the other hand, if we consider the restriction $\hat{\varphi} \colon A \to {\rm Im}(\varphi)$ and the inclusion $\iota \colon {\rm Im}(\varphi) \to M$ (so that $\varphi = \iota \, \hat{\varphi}$), one can note that $\iota$ is an $n$-absolutely pure precover of $M$, and so an $n$-absolutely pure cover by \cite[Ex. 2 of \S 5.2]{EJ00}. 

\item (f) $\Rightarrow$ (b): Consider an epimorphism $g \colon A \to M$ in $\Rmod$ where $A \in \mathsf{A}_n$. By (f), consider also a monic $n$-absolutely pure cover $\varphi \colon A' \to M$ (with $A' \in \mathsf{A}_n$). There exists a morphism $f \colon A \to A'$ such that $\varphi \, f = g$. Since $g$ is epic, so is $\varphi$. Thus, $\varphi$ is an isomorphism, and hence $M$ is $n$-absolutely pure.  

\item (a) $\Rightarrow$ (g): Immediate. 

\item (g) $\Rightarrow$ (a): Let $0 \to K \to R^m \to F \to 0$ be an exact sequence where $F \in \mathsf{FP}^{\leq n}$ and $K$ is finitely generated with ${\rm pd}(K) \leq n-1$. Since $R$ is Lee $n$-coherent, we have that $K$ is finitely presented. On the other hand, since ${\rm fd}(F) \leq 1$, it follows that $K$ is also flat, and hence $K$ is projective by \cite[Prop. 3.2.12]{EJ00}.

\item (g) $\Leftrightarrow$ (h) $\Leftrightarrow$ (i): Follows by dimension shifting and the fact that $\mathsf{F}_n$ contains the class of flat $R^{\rm op}$-modules. 

\item (g) $\Rightarrow$ (j): Immediate.

\item (j) $\Rightarrow$ (i): Follows by noticing that $\mathsf{F}_1$ is always closed under submodules, since it is the left Tor-orthogonal complement of $\mathsf{FP}^{\leq 1}$. 

\item (i) $\Leftrightarrow$ (k): $\mathsf{F}_n$ is always covering, and if $R$ is Lee $n$-coherent, we also have that $\mathsf{F}_n$ is closed under direct products. From its definition, note also that $1$-cotilting classes are always closed under submodules. Hence, the equivalence between (i) and (k) follows by \cite[Thm. 8.2.3]{GT}. 

\item (i) $\Leftrightarrow$ (l): Similar to (b) $\Leftrightarrow$ (f).
\end{itemize}
\end{proof}

\begin{remark} ~\
\begin{enumerate}[(1)]
\item In \cite[Prop. 3.6]{YangLiu11}, the authors claim that assertions (b), (f), (i) and (l) are equivalent. Regarding \cite[proof of (4) $\Rightarrow$ (1) in Prop. 3.6]{YangLiu11}, one can find the same gap pointed out in Remark \ref{rmk0} on \cite[Lem. 2.1]{YangLiu11}, concerning pullback diagrams from two monomorphisms (concretely, the morphism $M \to E^+$ appearing in their pullback diagram is not necessarily epic). The proof given for Proposition \ref{prop:injectives_closed_under_quotients} fixes this gap. 

\item An analog of Proposition \ref{prop:injectives_closed_under_quotients} is also valid if we replace the classes $\mathsf{FP}^{\leq n}$, $\mathsf{A}_n$ and $\mathsf{F}_n$ by $\mathsf{CFP}^{\leq n}$, $\mathsf{BA}_n$ and $\mathsf{BF}_n$, and Lee $n$-coherent rings by weakly Lee $n$-coherent rings, respectively. 
\end{enumerate}
\end{remark}

%%%%%%%%%%%%%%%%%%%%%%%%%%%%%%%%%%%%%
%%%%%%%%%%%%%%%%%%%%%%%%%%%%%%%%%%%%%
%%%%%%%%%%%%%%%%%%%%%%%%%%%%%%%%%%%%%
%%%%%%%%%%%%%%%%%%%%%%%%%%%%%%%%%%%%%

\section{Induced $n$-absolutely pure and $n$-flat modules and chain complexes} \label{sec:examples}

The purpose of this section is to show how to obtain new finitely presented modules with bounded projective dimension, new $n$-absolutely pure and new $n$-flat modules or complexes from old ones. 

%%%%%%%%%%%%%%%%%%%%%%%%%%%%%%%%%%%%%
%%%%%%%%%%%%%%%%%%%%%%%%%%%%%%%%%%%%%

\subsection*{Transfer properties}

Let us consider a ring homomorphism $\varphi \colon R \to S$, where $R$ and $S$ are commutative rings. In this subsection we study how to transfer modules from the classes $\mathsf{FP}^{\leq n}$, $\mathsf{A}_n$ and $\mathsf{F}_n$ in $\Rmod$ to their corresponding classes in ${\rm Mod}(S)$, and viceversa. To avoid confusion, we will use the notations $\mathsf{FP}^{\leq n}_R$, $\mathsf{A}^R_n$ and $\mathsf{F}^R_n$ whenever we consider these classes in $\Rmod$, or $\mathsf{FP}^{\leq n}_S$, $\mathsf{A}^S_n$ and $\mathsf{F}^S_n$ when considered in ${\rm Mod}(S)$.

From $\varphi \colon R \to S$, there is an associated adjoint triple $(\varphi^\ast,\varphi_\ast,\varphi^!)$ of additive functors, where:
\begin{itemize}
\item $\varphi^\ast \colon \Rmod \to {\rm Mod}(S)$ is the \emph{extension of scalars} functor given by
\[
\varphi^\ast(M) := S \otimes_R M,
\]
for every $M \in \Rmod$;

\item $\varphi_\ast \colon {\rm Mod}(S) \to \Rmod$ is the \emph{restriction of scalars} functor, where for every $N \in {\rm Mod}(S)$, $\varphi_\ast(N)$ is the abelian group $N$ with the $R$-module structure given by 
\[
r \cdot x := \varphi(r) \cdot x
\]
for every $r \in R$ and $x \in N$; and 

\item $\varphi^! \colon \Rmod \to {\rm Mod}(S)$ is the \emph{coextension of scalars} functor given by
\[
\varphi^!(M) := \Hom_R(S,M),
\]
for every $M \in \Rmod$.
\end{itemize}
Note that $\varphi^\ast$ (resp., $\varphi^!$) is right (resp., left) exact since it is a left (resp., right) adjoint. Then, $\varphi_\ast$ is exact. Moreover, $\varphi_\ast$ is always faithful, and it is also full if, and only if, $\varphi$ is a ring epimorphism (see for instance \cite[Ex. 4.5.14 (iv)]{Riehl16}).

\begin{proposition}\label{prop:transfer1}
The following assertions hold for any ring homomorphism $\varphi \colon R \to S$:
\begin{enumerate}[(1)]
\item If $\varphi$ makes $S$ a faithfully flat $R$-module, then $\varphi^\ast(F) \in \mathsf{FP}^{\leq n}_S$ for every $F \in \mathsf{FP}^{\leq n}_R$.

\item If $\varphi$ makes $S$ a finitely generated projective $R$-module, then $\varphi_\ast(F) \in \mathsf{FP}^{\leq n}_R$ for every $F \in \mathsf{FP}^{\leq n}_S$. 

\item If $\varphi$ is an epimorphism that makes $S$ a (nonzero) finitely generated free $R$-module, then $F \in \mathsf{FP}^{\leq n}_S$ whenever $\varphi_\ast(F) \in \mathsf{FP}^{\leq n}_R$. \\
\end{enumerate}
\end{proposition}

\begin{proof} ~\
\begin{enumerate}[(1)]
\item By \cite[Thm. 2.1.9]{Glaz89}, $\varphi^\ast(F)$ is finitely presented in ${\rm Mod}(S)$. On the other hand, there is an exact sequence
\[
0 \to P_n \to P_{n-1} \to \cdots \to P_1 \to P_0 \to F \to 0
\]
in $\Rmod$ with $P_k$ a projective $R$-module for every $0 \leq k \leq n$. Since $S$ is flat as an $R$-module, we get an exact sequence 
\[
0 \to S \otimes_R P_n \to S \otimes_R P_{n-1} \to \cdots \to S \otimes_R P_1 \to S \otimes_R P_0 \to S \otimes_R F \to 0.
\]
Moreover, note that each $S \otimes_R P_k$ is a projective $S$-module since $\Hom_S(S \otimes_R P_k,\sim)$ is an exact functor, being naturally isomorphic to $\Hom_R(P_k,\varphi_\ast(\sim))$ (a composition of two exact functors). Hence, ${\rm pd}(S \otimes_R F) \leq n$. 

\item We first show that $\varphi_\ast(F)$ is finitely presented over $R$. Since $F$ is finitely presented over $S$, we have an exact sequence
\[
S^{m_1} \to S^{m_0} \to F \to 0,
\]
with $m_0$ and $m_1$ nonnegative integers. Since $\varphi_\ast$ is an additive exact functor, we get the exact sequence
\[
[\varphi_\ast (S)]^{m_1} \to [\varphi_\ast (S)]^{m_0} \to \varphi_\ast (F) \to 0
\]
in $\Rmod$, where $[\varphi_\ast (S)]^{m_0}$ and $[\varphi_\ast (S)]^{m_1}$ are finitely generated and projective, since so is $\varphi_\ast (S)$. Hence, $\varphi_\ast (F)$ is finitely presented over $R$. 

Now consider an exact sequence
\[
0 \to P_n \to P_{n-1} \to \cdots \to P_1 \to P_0 \to F \to 0
\]
in ${\rm Mod}(S)$ with $P_k$ projective for every $0 \leq k \leq n$. This yields the following exact sequence in $\Rmod$:
\[
0 \to \varphi_\ast (P_n) \to \varphi_\ast (P_{n-1}) \to \cdots \to \varphi_\ast (P_1) \to \varphi_\ast (P_0) \to \varphi_\ast (F) \to 0.
\]
We can note that each $\varphi_\ast (P_k)$ is a projective $R$-module. Indeed, the functor $\Hom_R(\varphi_\ast (P_k), \sim)$ is naturally isomorphic to $\Hom_S(P_k,\Hom_R(S,\sim))$, a composition of exact functors. Hence, ${\rm pd}(\varphi_\ast (F)) \leq n$. 

\item Now assume that $\varphi$ is a ring epimorphism and that $S$ is finitely generated and free as an $R$-module. Then, $\varphi_\ast$ is full and faithful, and hence by \cite[Prop. 3.4.1]{BorceuxI} the functors $\varphi^\ast \circ \varphi_\ast$ and ${\rm id}_{S\mbox{-}\textrm{Mod}}$ are naturally isomorphic. So if we let $F \in S\mbox{-}\textrm{Mod}$ such that $\varphi_\ast(F) \in \mathsf{FP}^{\leq n}_R$, then by (1) we get $F \simeq \varphi^\ast(\varphi_\ast(F)) \in \mathsf{FP}^{\leq n}_S$ (note that $S$ is faithfully flat over $R$, since $S$ is a nonzero free $R$-module). 
\end{enumerate}
\end{proof}

\begin{proposition}\label{prop:transfer2}
The following assertions hold for any ring homomorphism $\varphi \colon R \to S$:
\begin{enumerate}[(1)]
\item If $\varphi$ makes $S$ a faithfully flat $R$-module, then:
\begin{enumerate}[(a)]
\item $\varphi_\ast(N) \in \mathsf{F}^R_n$ for every $N \in \mathsf{F}^S_n$. 

\item $\varphi_\ast(M) \in \mathsf{A}^R_n$ for every $M \in \mathsf{A}^S_n$.
\end{enumerate}

\item If $\varphi$ makes $S$ a finitely generated and projective $R$-module, then: 
\begin{enumerate}[(a)]
\item $\varphi^\ast(N) \in \mathsf{F}^S_n$ for every $N \in \mathsf{F}^R_n$. 

\item $\varphi^!(M) \in \mathsf{A}^S_n$ for every $M \in \mathsf{A}^R_n$. 
\end{enumerate}

\item If $\varphi$ is a ring epimorphism that makes $S$ a finitely generated and projective $R$-module, then: 
\begin{enumerate}[(a)]
\item $N \in \mathsf{F}^S_n$ whenever $\varphi_\ast(N) \in \mathsf{F}^R_n$. 

\item $M \in \mathsf{A}^S_n$ whenever $\varphi_\ast(M) \in \mathsf{A}^R_n$. \\
\end{enumerate}
\end{enumerate}
\end{proposition}

\begin{proof} {} ~\
\begin{enumerate}[(1)]
\item Part (a) (resp., (b)) follows by \cite[Coroll. 10.72]{RotmanBook} (resp., \cite[Coroll. 10.74]{RotmanBook}) and Proposition \ref{prop:transfer1} (1). 

\item Part (a) (resp., (b)) follows by \cite[Coroll. 10.72]{RotmanBook} (resp., \cite[Coroll. 10.75]{RotmanBook}) and Proposition \ref{prop:transfer1} (2). 

\item For part (a), if $N \in {\rm Mod}(S)$ is such that $\varphi_\ast(N) \in \mathsf{F}^R_n$ and $F \in \mathsf{FP}^{\leq n}_S$, then 
\[
\Tor^S_1(N,F) \cong \Tor^S_1(N,\varphi^\ast(\varphi_\ast(F))) \cong  \Tor^R_1(\varphi_\ast(N),\varphi_\ast(F))
\] 
since $\varphi$ is epic, and also by \cite[Coroll. 10.72]{RotmanBook}. By Proposition \ref{prop:transfer1} (2), we have that $\varphi_\ast(F) \in \mathsf{FP}^{\leq n}_R$, and so $\Tor^R_1(\varphi_\ast(N),\varphi_\ast(F)) = 0$. Hence, $N \in \mathsf{F}^S_n$. 

Regarding part (b), note that since $(\varphi_\ast,\varphi^!)$ is an adjoint pair where $\varphi_\ast$ is full and faithful, we have by the dual of \cite[Prop. 3.4.1]{BorceuxI} (see also \cite[Ex. III 36. (a)]{AssemBook}) that $\varphi^! \circ \varphi_\ast$ and ${\rm id}_{{\rm Mod}(S)}$ are naturally isomorphic. Hence, since $\varphi_\ast(M) \in \mathsf{A}^R_n$, we obtain by part (2) (b) that $M \simeq \varphi^!(\varphi_\ast(M)) \in \mathsf{A}^S_n$. 
\end{enumerate}
\end{proof}

%%%%%%%%%%%%%%%%%%%%%%%%%%%%%%%%%%%%%
%%%%%%%%%%%%%%%%%%%%%%%%%%%%%%%%%%%%%

\subsection*{$\bm{n}$-Absolutely pure and $\bm{n}$-flat modules over triangular matrix rings}

We aim to describe, over a triangular matrix ring 
\[
T = \left( \begin{array}{cc} A & 0 \\ U & B \end{array} \right),
\] 
how are the $T$-modules in $\mathsf{FP}^{\leq n}_T$, $\mathsf{A}_n^T$ and $\mathsf{F}_n^T$ in terms of the $A$-modules and $B$-modules in the corresponding classes $\mathsf{FP}^{\leq n}_A$, $\mathsf{FP}^{\leq n}_B$, $\mathsf{A}_n^A$, $\mathsf{A}_n^B$, $\mathsf{F}_n^A$ and $\mathsf{F}_n^B$. Here, $A$ and $B$ are associative rings with identity, and $U$ is a $(B,A)$-bimodule. It will be helpful to recall that the category ${\rm Mod}(T)$ of $T$-modules is equivalent to the category whose objects are given by pairs $M = (M_1, M_2)_\varphi$ where $M_1 \in {\rm Mod}(A)$, $M_2 \in {\rm Mod}(B)$ and $\varphi = \varphi^M \colon U \otimes_A M_1 \to M_2$ is a $B$-homomorphism. A morphism $f \colon (M_1, M_2)_\varphi \to (N_1, N_2)_\psi$ is given by a pair $f = (f_1,f_2)$ where $f_1 \colon M_1 \to N_1$ is an $A$-homomorphism and $f_2 \colon M_2 \to N_2$ is a $B$-homomorphism, such that the following diagram commutes:
\[
\begin{tikzpicture}[description/.style={fill=white,inner sep=2pt}] 
\matrix (m) [matrix of math nodes, row sep=2em, column sep=4em, text height=1.25ex, text depth=0.25ex] 
{ 
U \otimes_A M_1 & U \otimes_A N_1 \\ M_2 & N_2 \\
}; 
\path[->] 
(m-1-1) edge node[above] {\footnotesize$U \otimes f_1$} (m-1-2) edge node[left] {\footnotesize$\varphi^M$} (m-2-1) (m-1-2) edge node[right] {\footnotesize$\psi^N$} (m-2-2) (m-2-1) edge node[below] {\footnotesize$f_2$} (m-2-2)
;
\end{tikzpicture} 
\]

\begin{proposition}
Let $M = (M_1,M_2)_{\varphi} \in {\rm Mod}(T)$. The following assertions hold:
\begin{enumerate}[(1)]
\item If $M \in \mathsf{FP}^{\leq n}_T$, then $M_1$ is a finitely presented $A$-module with ${\rm pd}(M_1) \leq n$ and $M_2 / {\rm Im}(\varphi)$ is a finitely presented $B$-module. If in addition $\varphi$ is monic, ${}_B U$ is projective and $U_A$ is flat, then ${\rm pd}(M_2 / {\rm Im}(\varphi)) \leq n+1$ (and so $M_2 / {\rm Im}(\varphi) \in \mathsf{FP}^{\leq n+1}_B$).

\item If $\varphi$ is monic, ${}_B U$ is projective, $U_A$ is flat, $M_1 \in \mathsf{FP}^{\leq n}_A$ and $M_2 / {\rm Im}(\varphi) \in \mathsf{FP}^{\leq n}_B$, then $M \in \mathsf{FP}^{\leq n+1}_T$. 

\item If $M_2 \in \mathsf{FP}^{\leq n}_B$, then $(0,M_2)_{0} \in \mathsf{FP}^{\leq n}_T$. 

\item If $M_1 \in \mathsf{FP}^{\leq n}_A$, then $(M_1, U \otimes_A M_1)_{{\rm id}} \in \mathsf{FP}^{\leq n+1}_T$, provided that ${}_B U$ is projective and $U_A$ is flat. \\
\end{enumerate}
\end{proposition}

\begin{proof} 
In what follows, it will be important to keep in mind the following results:
\begin{itemize}
\item \cite[Ex. 1C-(b)]{Goodearl}: $(F_1,F_2)_{\varphi^F}$ is finitely generated in ${\rm Mod}(T)$ if, and only if, $F_1$ is finitely generated in ${\rm Mod}(A)$ and $F_2 / {\rm Im}(\varphi^F)$ is finitely generated in ${\rm Mod}(B)$. 

\item \cite[Thm. 3.1]{HaghanyVaradarajan99}: $(F_1,F_2)_{\varphi^F}$ is projective in ${\rm Mod}(T)$ if, and only if, $F_1$ is projective in ${\rm Mod}(A)$, $F_2 / {\rm Im}(\varphi^F)$ is projective in ${\rm Mod}(B)$ and $\varphi^F$ is monic.

\item \cite[Coroll. 3.3]{Mao20}: For every $(F_1,F_2)_{\varphi^F}$ in ${\rm Mod}(T)$, one has that
\[
\max \{ \pd(F_1), \pd(F_2) \} \leq \pd((F_1,F_2)_{\varphi^F}) \leq \max\{ \pd(F_1) + 1, \pd(F_2) \}.
\]
\end{itemize}
\begin{enumerate}[(1)]
\item Let $(F^1_1,F^1_2)_{\varphi^1} \xrightarrow{(f^1_1,f^1_2)} (F^0_1,F^0_2)_{\varphi^0} \xrightarrow{(f^0_1,f^0_2)} (M_1,M_2) \to 0$ be an exact sequence in ${\rm Mod}(T)$ where $(F^i_1,F^i_2)_{\varphi^i}$ is finitely generated and projective for $i = 0,1$. Then each $F^i_1$ is a finitely generated and projective $A$-module, $F^i_2 / {\rm Im}(\varphi^i)$ is a finitely generated and projective $B$-module, and $\varphi^i$ is monic. In particular, the exact sequence $F^1_1 \to F^0_1 \to M_1 \to 0$ implies that $M_1$ is a finitely presented $A$-module. We also have the following commutative diagram with exact rows and columns:
\[
\begin{tikzpicture}[description/.style={fill=white,inner sep=2pt}] 
\matrix (m) [matrix of math nodes, row sep=2em, column sep=4em, text height=1.25ex, text depth=0.25ex] 
{ 
0 & 0 & {} & {} \\
U \otimes_A F^1_1 & U \otimes_A F^0_1 & U \otimes_A M_1 & 0 \\
F^1_2 & F^0_2 & M_2 & 0 \\ 
F^1_2 / {\rm Im}(\varphi^1) & F^0_2 / {\rm Im}(\varphi^0) & M_2 / {\rm Im}(\varphi) & 0 \\
0 & 0 & 0 & {} \\
}; 
\path[->] 
(m-1-1) edge (m-2-1) (m-1-2) edge (m-2-2)
(m-2-1) edge node[above] {\footnotesize$U \otimes f^1_1$} (m-2-2) edge node[left] {\footnotesize$\varphi^1$} (m-3-1) (m-2-2) edge node[above] {\footnotesize$U \otimes f^0_1$} (m-2-3) edge node[left] {\footnotesize$\varphi^0$} (m-3-2) (m-2-3) edge (m-2-4) edge node[left] {\footnotesize$\varphi$} (m-3-3)
(m-3-1) edge node[above] {\footnotesize$f^1_2$} (m-3-2) edge (m-4-1) (m-3-2) edge node[above] {\footnotesize$f^0_2$} (m-3-3) edge (m-4-2) (m-3-3) edge (m-3-4) edge (m-4-3)
(m-4-1) edge (m-4-2) edge (m-5-1) (m-4-2) edge (m-4-3) edge (m-5-2) (m-4-3) edge (m-4-4) edge (m-5-3)
;
\end{tikzpicture} 
\]
This implies that $M_2 / {\rm Im}(\varphi)$ is finitely presented. On the other hand, there is an exact sequence
\[
0 \to (P^n_1, P^n_2)_{\varphi^n} \to \cdots \to (P^1_1,P^1_2)_{\varphi^1} \to (P^0_1,P^0_2)_{\varphi^0} \to (M_1,M_2) \to 0
\]
in ${\rm Mod}(T)$ with $(P^i_1,P^i_2)$ projective for every $0 \leq i \leq n$. Then, we have the exact sequence
\[
0 \to P^n_1 \to \cdots \to P^1_1\to P^0_1 \to M_1 \to 0
\]
where each $P^i_1$ is a projective $A$-module, and so, ${\rm pd}(M_1) \leq n$. Hence, $M_1 \in \mathsf{FP}^{\leq n}_A$. 

Now suppose that $\varphi$ is monic, ${}_B U$ is projective and $U_A$ is flat. From the previous diagram and the assumptions, we have that an exact sequence
\[
0 \to U \otimes_A M_1 \xrightarrow{\varphi} M_2 \to M_2 / {\rm Im}(\varphi) \to 0,
\]
where ${\rm pd}(M_2) \leq n$. Also, we have the following exact sequence in ${\rm Mod}(B)$:
\[
0 \to U \otimes_A P^n_1 \to \cdots \to U \otimes_A P^1_1 \to U \otimes_A P^0_1 \to U \otimes_A M_1 \to 0
\]
We claim that each $U \otimes_A P^i_1$ is a projective $B$-module. Indeed, we have a section $P^i_1 \to A^{(J_i)}$ where $A^{(J_i)}$ is a free $A$-module, which yields a section $U \otimes_A P^i_1 \to U \otimes_A A^{(J_i)}$. On the other hand, $U \otimes_A A^{(J_i)} \simeq U^{(J_i)}$ where ${}_B U$ is projective, and so $U \otimes_A P^i_1$ is projective in ${\rm Mod}(B)$. It follows that $\pd(U \otimes_A M_1) \leq n$. Therefore, $\pd(M_2 / {\rm Im}(\varphi)) \leq n + 1$.

\item We first show that $\pd(M) \leq n+1$. Note that it suffices to show that $\pd(M_2) \leq n$, but this follows from the short exact sequence 
\[
0 \to U \otimes_A M_1 \xrightarrow{\varphi} M_2 \to M_2 / {\rm Im}(\varphi) \to 0.
\]
Indeed, as in the previous part, we can note that $\pd(U \otimes_A M_1) \leq n$, and so $\pd(M_2) \leq n$ since $\pd(M_2 / {\rm Im}(\varphi)) \leq n$. 

Now let us show that $M$ is finitely presented. Note that $M$ is finitely generated, and so there is a short exact sequence
\[
0 \to (K_1,K_2)_{\varphi^K} \to (F_1,F_2)_{\varphi^F} \to (M_1,M_2)_{\varphi^M} \to 0
\]
where $(F_1,F_2)_{\varphi^F}$ is finitely generated and projective. On the one hand, we obtain a short exact sequence
\[
0 \to K_1 \to F_1 \to M_1 \to 0
\]
where $F_1$ is finitely generated and projective. Since $M_1$ is finitely presented, we have that $K_1$ is finitely generated. On the other hand, since $U_A$ is flat, we have the following commutative diagram with exact rows:
\[
\begin{tikzpicture}[description/.style={fill=white,inner sep=2pt}] 
\matrix (m) [matrix of math nodes, row sep=2em, column sep=2em, text height=1.25ex, text depth=0.25ex] 
{ 
0 & U \otimes_A K_1 & U \otimes_A F_1 & U \otimes_A M_1 & 0 \\
0 & K_2 & F_2 & M_2 & 0 \\ 
}; 
\path[->] 
(m-1-1) edge (m-1-2) (m-1-2) edge (m-1-3) edge node[right] {\footnotesize$\varphi^K$} (m-2-2) (m-1-3) edge (m-1-4) edge node[right] {\footnotesize$\varphi^F$} (m-2-3) (m-1-4) edge (m-1-5) edge node[right] {\footnotesize$\varphi^M$} (m-2-4)
(m-2-1) edge (m-2-2) (m-2-2) edge (m-2-3) (m-2-3) edge (m-2-4) (m-2-4) edge (m-2-5)
;
\end{tikzpicture} 
\]
Since $\varphi^F$ and $\varphi^M$ are monic, the snake lemma implies that $\varphi^K$ is monic and that there is a short exact sequence of the form
\[
0 \to K_2 / {\rm Im}(\varphi^K) \to F_2 / {\rm Im}(\varphi^F) \to M_2 / {\rm Im}(\varphi^M) \to 0,
\]
where $M_2 / {\rm Im}(\varphi^M)$ is finitely presented and $F_2 / {\rm Im}(\varphi^F)$ is finitely generated and projective. It follows that $K_2 / {\rm Im}(\varphi^K)$ is finitely generated, and hence $(K_1,K_2)_{\varphi^K}$ is finitely generated. Therefore, $(M_1,M_2)_{\varphi^M}$ is finitely presented. 

\item First, note that since $n$ is a positive integer, we have that 
\[
\pd((0,M_2)_0) \leq \max\{ 1, \pd(M_2) \} \leq n.
\]
Note also that $(0,M_2)_0$ is finitely generated, and so there is an exact sequence
\[
0 \to (K_1,K_2)_{\varphi^K} \to (F_1,F_2)_{\varphi^F} \to (0,M_2)_{0} \to 0
\]
where $(F_1,F_2)_{\varphi^F}$ is finitely generated and projective, and $K_1 \simeq F_1$. Thus, we have the following commutative diagram with exact rows:
\[
\begin{tikzpicture}[description/.style={fill=white,inner sep=2pt}] 
\matrix (m) [matrix of math nodes, row sep=2em, column sep=2em, text height=1.25ex, text depth=0.25ex] 
{ 
0 & U \otimes_A K_1 & U \otimes_A F_1 & 0 & 0 \\
0 & K_2 & F_2 & M_2 & 0 \\ 
}; 
\path[->] 
(m-1-1) edge (m-1-2) (m-1-2) edge node[above] {\footnotesize$\simeq$} (m-1-3) edge node[right] {\footnotesize$\varphi^K$} (m-2-2) (m-1-3) edge (m-1-4) edge node[right] {\footnotesize$\varphi^F$} (m-2-3) (m-1-4) edge (m-1-5) edge (m-2-4)
(m-2-1) edge (m-2-2) (m-2-2) edge (m-2-3) (m-2-3) edge (m-2-4) (m-2-4) edge (m-2-5)
;
\end{tikzpicture} 
\]
Since $\varphi^F$ is monic, so is $\varphi^K$, and so we get the exact sequence
\[
0 \to K_2 / {\rm Im}(\varphi^K) \to F_2 / {\rm Im}(\varphi^F) \to M_2 \to 0
\]
where $F_2 / {\rm Im}(\varphi^F)$ is finitely generated and projective and $M_2$ is finitely presented. Then, $K_2 / {\rm Im}(\varphi^K)$ and $K_1 \simeq F_1$ are finitely generated, that is, $(K_1,K_2)_{\varphi^K}$ is finitely generated, and hence $(0,M_2)_{0}$ is finitely presented. 

\item Firstly, note that
\[
\pd((M_1, U \otimes_A M_1)_{{\rm id}}) \leq \max\{\pd(M_1) + 1, \pd(U \otimes_A M_1)\},
\]
where $\pd(U \otimes_A M_1) \leq n$ by the arguments used in the proof of part (1). Then, $\pd((M_1, U \otimes_A M_1)_{{\rm id}}) \leq n+1$. On the other hand, we can also note that $(M_1, U \otimes_A M_1)_{{\rm id}}$ is finitely generated. Then, there is an exact sequence 
\[
0 \to (K_1,K_2)_{\varphi^K} \to (F_1,F_2)_{\varphi^F} \to (M_1, U \otimes_A M_1)_{{\rm id}} \to 0
\]
where $(F_1,F_2)_{\varphi^F}$ is finitely generated and projective, from which we can obtain the following commutative diagram with exact rows:
\[
\begin{tikzpicture}[description/.style={fill=white,inner sep=2pt}] 
\matrix (m) [matrix of math nodes, row sep=2em, column sep=2em, text height=1.25ex, text depth=0.25ex] 
{ 
0 & U \otimes_A K_1 & U \otimes_A F_1 & U \otimes_A M_1 & 0 \\
0 & K_2 & F_2 & U \otimes_A M_1 & 0 \\ 
}; 
\path[->] 
(m-1-1) edge (m-1-2) (m-1-2) edge (m-1-3) edge node[right] {\footnotesize$\varphi^K$} (m-2-2) (m-1-3) edge (m-1-4) edge node[right] {\footnotesize$\varphi^F$} (m-2-3) (m-1-4) edge (m-1-5) edge node[right] {\footnotesize${\rm id}$} (m-2-4)
(m-2-1) edge (m-2-2) (m-2-2) edge (m-2-3) (m-2-3) edge (m-2-4) (m-2-4) edge (m-2-5)
;
\end{tikzpicture} 
\]
By the snake lemma, we have that $K_2 / {\rm Im}(\varphi^K) \simeq F_2 / {\rm Im}(\varphi^F)$, and so $K_2 / {\rm Im}(\varphi^K)$ is finitely generated. Note that the exact sequence $0 \to K_1 \to F_1 \to M_1 \to 0$ implies that $K_1$ is also finitely generated. It then follows that $(K_1,K_2)_{\varphi^K}$ is finitely generated, and hence $(M_1,U \otimes_A M_1)_{{\rm id}}$ is finitely presented. 
\end{enumerate}
\end{proof}

The following is a consequence of the previous proposition and \cite[Lems. 3.2 \& 3.5]{Mao20}.

\begin{proposition}
The following assertions hold:
\begin{enumerate}[(1)]
\item $(N_1,N_2)_{\varphi^N} \in \mathsf{A}^T_n$ $\Rightarrow$ $N_2 \in \mathsf{A}^B_n$.

\item $(L_1,L_2)_{\varphi^L} \in \mathsf{F}^T_n$ $\Rightarrow$ $L_2 \in \mathsf{F}^B_n$.

\item $N_1 \in \mathsf{A}^A_n$ $\Rightarrow$ $(N_1,0)_{0} \in \mathsf{A}^T_n$.

\item $L_1 \in \mathsf{F}^{A^{\rm op}}_n$ $\Rightarrow$ $(L_1,0)_{0} \in \mathsf{F}^{T^{\rm op}}_n$.
\end{enumerate}
If in addition $U_A$ is flat, the following implication also holds:
\begin{enumerate}[(1)]
\setcounter{enumi}{4}
\item $L_1 \in \mathsf{F}^{A}_n$ $\Rightarrow$ $(L_1,U \otimes_A L_1)_{{\rm id}} \in \mathsf{F}^{T}_n$.
\end{enumerate}
If in addition ${}_BU$ is projective and $U_A$ is flat, the following implications also hold:
\begin{enumerate}[(1)]
\setcounter{enumi}{5}
\item $(N_1,N_2)_{\varphi^N} \in \mathsf{A}^T_{n+1}$ $\Rightarrow$ $N_1 \in \mathsf{A}^A_n$.

\item $(L_1,L_2)_{\varphi^L} \in \mathsf{F}^T_{n+1}$ $\Rightarrow$ $L_1 \in \mathsf{F}^A_n$.
\end{enumerate}
\end{proposition}

%%%%%%%%%%%%%%%%%%%%%%%%%%%%%%%%%%%%%
%%%%%%%%%%%%%%%%%%%%%%%%%%%%%%%%%%%%%

\subsection*{$\bm{n}$-Absolutely pure and $\bm{n}$-flat modules and chain complexes}

Let $\Rch$ and $\chR$ denote the categories of complexes of $R$-modules and $R^{\rm op}$-modules, and their corresponding chain maps, respectively. Objects in $\Rch$ are denoted as 
\[
X_\bullet = \cdots \to X_{m+1} \xrightarrow{\partial_{m+1}} X_m \xrightarrow{\partial_m} X_{m-1} \to \cdots
\]
where $X_m \in \Rmod$ and $\partial_{m} \circ \partial_{m+1} = 0$ for every $m \in \mathbb{Z}$. The kernels and images of the differential maps $\partial_m$ will be denoted by $Z_m(X_\bullet)$ and $B_{m-1}(X_\bullet)$, respectively. We use the same notations for complexes of $R^{\rm op}$-modules. The definitions given below will be stated mostly for $\Rch$, since they have their corresponding counterparts in $\chR$.

The definitions of finitely generated and finitely presented complexes can be traced back to Enochs and García Rozas work \cite[Def. 2.1]{EnochsGarciaRozas98}. Moreover, they characterized these complexes in \cite[Lem. 2.2]{EnochsGarciaRozas98}. Concretely, a chain complex $X_\bullet$ is finitely generated (resp., finitely presented) if, and only if, $X_\bullet$ is bounded, above and below, and each $X_m$ is a finitely generated (resp., finitely presented) $R$-module.

We recall also from García Rozas \cite{GarciaRozas99} that a chain complex $X_\bullet \in \Rch$ has projective dimension at most $n \in \mathbb{Z}_{\geq 0}$ if there exists an exact sequence 
\[
0 \to P^n_\bullet \to P^{n-1}_\bullet \to \cdots \to P^1_\bullet \to P^0_\bullet \to X_\bullet \to 0
\]
in $\Rch$ where each $P^k_\bullet$ is a projective complex. By \cite[dual of Thm. 3.1.3]{GarciaRozas99}, $\textrm{pd}(X_\bullet) \leq n$ if, and only if, $X_\bullet$ is exact and $\textrm{pd}(Z_m(X_\bullet)) \leq n$ for every $m \in \mathbb{Z}$. 

From now on, let us denote by $\mathsf{FP}^{\leq n}_{\textrm{Ch}}$ the class of finitely presented chain complexes with projective dimension at most $n$. From the previous descriptions, we can get the following characterization of $\mathsf{FP}^{\leq n}_{\textrm{Ch}}$.

\begin{proposition}\label{prop:FPn-complex}
A chain complex $F_\bullet$ belongs to $\mathsf{FP}^{\leq n}_{{\rm Ch}}$ if, and only if, $F_\bullet$ is bounded, exact and $F_m \in \mathsf{FP}^{\leq n}$ for every $m \in \mathbb{Z}$. 
\end{proposition}

\begin{proof}
Let us first suppose that $F_\bullet \in \mathsf{FP}^{\leq n}_{{\rm Ch}}$. In particular, $F_\bullet$ is finitely presented, and so it is bounded with each $F_m$ a finitely presented $R$-module. Without loss of generality, suppose $F_\bullet$ is of the form
\[
F_\bullet = \cdots \to 0 \to F_l \to F_{l-1} \to \cdots \to F_1 \to F_0 \to 0 \to \cdots.
\]
Since $F_\bullet$ has also projective dimension at most $n$, we have that $F_\bullet$ is exact and its kernels have projective dimension at most $n$. It then follows that $\textrm{pd}(F_0) \leq n$ and $\textrm{pd}(Z_1(F_\bullet)) \leq n$, which in turn implies that $\textrm{pd}(F_1) \leq n$. Following these procedure, we can conclude that $\textrm{pd}(F_m) \leq n$ for every $m \in \mathbb{Z}$. Hence, $F_\bullet$ is bounded, exact and with $F_m \in \mathsf{FP}^{\leq n}$ for every $m \in \mathbb{Z}$.

Now suppose that $F_\bullet$ is bounded, exact and with each $F_m \in \mathsf{FP}^{\leq n}$. Again, we may assume that $F_\bullet$ is of the form
\[
F_\bullet = \cdots \to 0 \to F_l \to F_{l-1} \to \cdots \to F_1 \to F_0 \to 0 \to \cdots.
\]
Then, clearly $F_\bullet$ is finitely presented. On the other hand, from the short exact sequence $0 \to Z_1(F_\bullet) \to F_1 \to F_0 \to 0$ we get that $\textrm{pd}(Z_1(F_\bullet)) \leq n$. Repeating this procedure yields that $\textrm{pd}(Z_m(F_\bullet)) \leq n$ for every $m \in \mathbb{Z}$. Therefore, $\textrm{pd}(F_\bullet) \leq n$. 
\end{proof}

We can define $n$-absolutely pure and $n$-flat chain complexes in the same way as their module counterparts, that is, as Ext and Tor orthogonal complements of $\mathsf{FP}^{\leq n}_{{\rm Ch}}$. Recall from \cite{GarciaRozas99} that given $Y_\bullet \in \chR$ and $X_\bullet \in \Rch$, their tensor product $Y_\bullet \otimes X_\bullet$ is defined as the chain complex of abelian groups
\[
(Y_\bullet \otimes X_\bullet)_m = \bigoplus_{k \in \mathbb{Z}} (Y_k \otimes_R X_{m-k})
\]
with differentials given by 
\[
\partial(y \otimes x) = \partial(y) \otimes x + (-1)^{|y|} y \otimes \partial(x)
\]
where $|y|$ denotes the degree of $y$. On the other hand, the modified tensor product $Y_\bullet \overline{\otimes} X_\bullet$ is defined as the chain complex of abelian groups
\[
(Y_\bullet \overline{\otimes} X_\bullet)_m = \frac{(Y_\bullet \otimes X_\bullet)_m}{B_m(Y_\bullet \otimes X_\bullet)}
\]
with differentials given by 
\[
\partial(\overline{y \otimes x}) = \overline{\partial(y) \otimes x}.
\]
The left derived functors of the modified tensor product will be denoted by 
\[
\overline{{\rm Tor}}_i(-,\sim) \colon \chR \times \Rch \to {\rm Ch}(\mathbb{Z}\mbox{-}{\rm Mod}).
\] 
The right derived extension functors of the categorical ${\rm Hom}(-,\sim)$ in $\Rch$, on the other hand, are defined in the same way as for modules, and will be denoted by
\[
{\rm Ext}^i(-,\sim) \colon \Rch \times \Rch \to {\rm Ch}(\mathbb{Z}\mbox{-}{\rm Mod}).
\]

\mbox{}

\begin{definition}
Let $M_\bullet \in \Rch$ and $N_\bullet \in \chR$. We say that:
\begin{enumerate}[(1)]
\item $M_\bullet$ is \textbf{$\bm{n}$-absolutely pure} if ${\rm Ext}^i(F_\bullet,M_\bullet) = 0$ for every $F_\bullet \in \mathsf{FP}^{\leq n}_{{\rm Ch}}$;

\item $N_\bullet$ is \textbf{$\bm{n}$-flat} if $\overline{{\rm Tor}}_i(N_\bullet,F_\bullet) = 0$ for every $F_\bullet \in \mathsf{FP}^{\leq n}_{{\rm Ch}}$. \\
\end{enumerate}
\end{definition}

Chain complex versions of $n$-absolutely pure and $n$-flat modules are also defined and studied in \cite{SelvarajSaravanan16}. The approach given there is different to our previous definition, since the authors consider the projective dimension of chain complexes in the sense of Avramov and Foxby \cite{AvramovFoxby91}, and not García Rozas'. Then, the characterizations of $n$-absolutely pure\footnote{These complexes are called $n$-FP-injective in \cite{SelvarajSaravanan16}.} and $n$-flat complexes obtained in \cite{SelvarajSaravanan16} differ from ours, presented below.

\begin{proposition}\label{prop:characterization-nabs_complex}
A chain complex $M_\bullet \in \Rch$ is $n$-absolutely pure if, and only if, $M_m \in \mathsf{A}_n$ for every $m \in \mathbb{Z}$. 
\end{proposition}

\begin{proof}
Let us first assume that $M_\bullet$ is $n$-absolutely pure, and consider $F \in \mathsf{FP}^{\leq n}$. By Proposition \ref{prop:FPn-complex},  $D^m(F) \in \mathsf{FP}^{\leq n}_{{\rm Ch}}$ for every $m \in \mathbb{Z}$, where $D^m(F)$ is the disk complex with $F$ at degrees $m$ and $m-1$, zero at the rest of the degrees, and where the only nonzero differential map is the identity $F \to F$. By \cite[Lem. 3.1 (5)]{GillespieFlatModel}, we have that a natural isomorphism $0 = \Ext^1(D^m(F),M_\bullet) \cong \Ext^1_R(F,M_m)$. Since $F \in \mathsf{FP}^{\leq n}$ is arbitrary, we have that $M_m \in \mathsf{A}_n$ for every $m \in \mathbb{Z}$. 

Now suppose that $M_\bullet$ is a complex with components in $\mathsf{A}_n$, and let $F_\bullet \in \mathsf{FP}^{\leq n}_{{\rm Ch}}$. Without loss of generality, suppose that $F_\bullet$ is of the form
\[
F_\bullet = \cdots \to 0 \to F_m \to F_{m-1} \to \cdots \to F_1 \to F_0 \to 0 \to \cdots.
\]
We use induction on $m$ to show that $\Ext^1(F_\bullet, M_\bullet) = 0$.
\begin{itemize}
\item Base case $m = 1$: In this case, $F_\bullet$ is isomorphic to a disk complex $D^1(F)$ for some $F \in \mathsf{FP}^{\leq n}$. Thus, we get $\Ext^1(F_\bullet, M_\bullet) \cong \Ext^1(D^1(F),M_\bullet) \cong \Ext^1_R(F,M_1) = 0$.

\item Suppose that $\Ext^1(F'_\bullet,M_\bullet) = 0$ for every $F'_\bullet \in \mathsf{FP}^{\leq n}_{{\rm Ch}}$ of the form
\[
F'_\bullet = \cdots \to 0 \to F'_{m-1} \to \cdots \to F'_1 \to F'_0 \to 0 \to \cdots,
\]
and let $F_\bullet$ as above. We can form the following commutative diagram with exact rows and columns:
\[
\begin{tikzpicture}[description/.style={fill=white,inner sep=2pt}] 
\matrix (m) [matrix of math nodes, row sep=1.5em, column sep=1.5em, text height=1.25ex, text depth=0.25ex] 
{ 
0 & \cdots & 0 & 0 & 0 & 0 & \cdots & 0 & 0 & 0 & \cdots \\
D^m(F_m) \colon & \cdots & 0 & F_m & F_m & 0 & \cdots & 0 & 0 & 0 & \cdots \\
F_\bullet \colon & \cdots & 0 & F_m & F_{m-1} & F_{m-2} & \cdots & F_1 & F_0 & 0 & \cdots \\
F'_\bullet \colon & \cdots & 0 & 0 & Z_{m-2}(F_\bullet) & F_{m-2} & \cdots & F_1 & F_0 & 0 & \cdots \\
0 & \cdots & 0 & 0 & 0 & 0 & \cdots & 0 & 0 & 0 & \cdots \\
}; 
\path[->] 
(m-1-1) edge (m-2-1) (m-2-1) edge (m-3-1) (m-3-1) edge (m-4-1) (m-4-1) edge (m-5-1)
(m-1-3) edge (m-2-3) (m-1-4) edge (m-2-4) (m-1-5) edge (m-2-5) (m-1-6) edge (m-2-6) (m-1-8) edge (m-2-8) (m-1-9) edge (m-2-9) (m-1-10) edge (m-2-10)
(m-4-3) edge (m-5-3) (m-4-4) edge (m-5-4) (m-4-5) edge (m-5-5) (m-4-6) edge (m-5-6) (m-4-8) edge (m-5-8) (m-4-9) edge (m-5-9) (m-4-10) edge (m-5-10)
(m-2-2) edge (m-2-3) (m-2-3) edge (m-2-4) (m-2-5) edge (m-2-6) (m-2-6) edge (m-2-7) (m-2-7) edge (m-2-8) (m-2-8) edge (m-2-9) (m-2-9) edge (m-2-10) (m-2-10) edge (m-2-11)
(m-3-2) edge (m-3-3) (m-3-3) edge (m-3-4) (m-3-4) edge (m-3-5) (m-3-5) edge (m-3-6) (m-3-6) edge (m-3-7) (m-3-7) edge (m-3-8) (m-3-8) edge (m-3-9) (m-3-9) edge (m-3-10) (m-3-10) edge (m-3-11)
(m-4-2) edge (m-4-3) (m-4-3) edge (m-4-4) (m-4-4) edge (m-4-5) (m-4-5) edge (m-4-6) (m-4-6) edge (m-4-7) (m-4-7) edge (m-4-8) (m-4-8) edge (m-4-9) (m-4-9) edge (m-4-10) (m-4-10) edge (m-4-11)
(m-2-3) edge (m-3-3) (m-3-3) edge (m-4-3) (m-3-4) edge (m-4-4)
(m-2-5) edge (m-3-5) (m-3-5) edge (m-4-5)
(m-2-6) edge (m-3-6) 
(m-2-8) edge (m-3-8) (m-2-9) edge (m-3-9) (m-2-10) edge (m-3-10) (m-3-10) edge (m-4-10)
;
\path[-,font=\scriptsize]
(m-2-4) edge [double, thick, double distance=2pt] (m-2-5)
(m-2-4) edge [double, thick, double distance=2pt] (m-3-4)
(m-3-6) edge [double, thick, double distance=2pt] (m-4-6)
(m-3-8) edge [double, thick, double distance=2pt] (m-4-8)
(m-3-9) edge [double, thick, double distance=2pt] (m-4-9)
;
\end{tikzpicture} 
\]
From the short exact sequence $0 \to F_m \to F_{m-1} \to Z_{m-2}(F_\bullet) \to 0$, since $F_m$ and $F_{m-1}$ are finitely presented, we have that so is $Z_{m-2}(F_\bullet)$. Note also that $\textrm{pd}(Z_{m-2}(F_\bullet)) \leq n$. Thus, $F'_\bullet \in \mathsf{FP}^{\leq n}_{{\rm Ch}}$. From the short exact sequence $0 \to D^m(F_m) \to F_\bullet \to F'_\bullet \to 0$, we obtain the exact sequence $$\Ext^1(F'_\bullet,M_\bullet) \to \Ext^1(F_\bullet, M_\bullet) \to \Ext^1(D^m(F_m),M_\bullet)$$ of abelian groups where $\Ext^1(F'_\bullet,M_\bullet) = 0$ by the induction hypothesis, and $\Ext^1(D^m(F_m),M_\bullet) = 0$ by the base case. It follows that $\Ext^1(F_\bullet, M_\bullet) = 0$ for every $F_\bullet \in \mathsf{FP}^{\leq n}_{{\rm Ch}}$, and so $M_\bullet$ is $n$-absolutely pure. 
\end{itemize}
\end{proof}

From the previous result and \cite[Def. 3.1]{GillespieDegreewise}, we have that the class of $n$-absolutely pure complexes in $\Rch$ coincides with the class of degreewise $n$-absolutely pure complexes, denoted by ${\rm dw}\,\widetilde{\mathsf{A}_n}$. In order to obtain a similar description for $n$-flat complexes, we need a Tor analog for \cite[Lem. 3.1 (5)]{GillespieFlatModel}. We are not aware if this result appears in the literature, so we provide a proof for the reader's convenience. In what follows, given a chain complex $Y_\bullet \in \chR$, the \emph{$m$-th suspension} of $Y_\bullet$ is defined as the chain complex $\Sigma^m(Y_\bullet)$ given by 
\[
[\Sigma^m(Y_\bullet)]_k = Y_{k-m}
\] 
with differentials $\partial_k = (-1)^m\partial^{Y_\bullet}_{k-m}$ for every $k \in \mathbb{Z}$.

\begin{lemma}\label{lem:TorAnalog}
For every $Y_\bullet \in \chR$ and $N \in \Rmod$, there is a natural isomorphism 
\[
\overline{{\rm Tor}}_i(Y_\bullet,D^m(N)) \cong {\rm Tor}^R_i(\Sigma^m(Y_\bullet),N),
\]
where ${\rm Tor}^R_i(\Sigma^m(Y_\bullet),N)$ is the complex
\[
\cdots \to {\rm Tor}^R_i([\Sigma^m(Y_\bullet)]_{k+1},N) \to {\rm Tor}^R_i([\Sigma^m(Y_\bullet)]_{k},N) \to {\rm Tor}^R_i([\Sigma^m(Y_\bullet)]_{k-1},N) \to \cdots.
\]
\end{lemma}

\begin{proof}
Let $\cdots \to P^2_\bullet \to P^1_\bullet \to P^0_\bullet \to Y_\bullet \to 0$ be a projective resolution of $Y_\bullet$. Recall that $\overline{{\rm Tor}}_i(Y_\bullet,D^m(N))$ is the $i$-th homology group of the complex 
\[
\cdots \to P^2_\bullet \overline{\otimes} D^m(N) \to P^1_\bullet \overline{\otimes} D^m(N) \to P^0_\bullet \overline{\otimes} D^m(N) \to 0.
\]
On the other hand, by \cite[Prop. 4.2.1 (4)]{GarciaRozas99} we have that there is a natural isomorphism $P^i_\bullet \overline{\otimes} D^m(N) \cong \Sigma^m(P^i_\bullet) \otimes_R N$ for every $i \in \mathbb{Z}$. Thus, we have the following isomorphism in the category of complexes of chain complexes of abelian groups:
\[
\begin{tikzpicture}[description/.style={fill=white,inner sep=2pt}] 
\matrix (m) [matrix of math nodes, row sep=2.5em, column sep=1.5em, text height=1.25ex, text depth=0.25ex] 
{ 
\cdots & P^2_\bullet \overline{\otimes} D^m(N) & P^1_\bullet \overline{\otimes} D^m(N) & P^0_\bullet \overline{\otimes} D^m(N) & 0 \\ 
\cdots & \Sigma^m(P^2_\bullet) \otimes_R N & \Sigma^m(P^1_\bullet) \otimes_R N & \Sigma^m(P^0_\bullet) \otimes_R N & 0 \\
}; 
\path[->] 
(m-1-1) edge (m-1-2) (m-1-2) edge (m-1-3) (m-1-3) edge (m-1-4) (m-1-4) edge (m-1-5)
(m-2-1) edge (m-2-2) (m-2-2) edge (m-2-3) (m-2-3) edge (m-2-4) (m-2-4) edge (m-2-5)
(m-1-2) edge node[right] {\footnotesize$\cong$} (m-2-2) (m-1-3) edge node[right] {\footnotesize$\cong$} (m-2-3) (m-1-4) edge node[right] {\footnotesize$\cong$} (m-2-4)
;
\end{tikzpicture} 
\]
This implies that the top and bottom complexes have naturally isomorphic homology groups, that is,  $\overline{{\rm Tor}}_i(Y_\bullet,D^m(N)) \cong {\rm Tor}^R_i(\Sigma^m(Y_\bullet),N)$.
\end{proof}

From the previous lemma, the proof of the following result is similar to the one shown in Proposition \ref{prop:characterization-nabs_complex}.

\begin{proposition}
A chain complex $N_\bullet \in \chR$ is $n$-flat if, and only if, $N_m \in \mathsf{F}_n$ for every $m \in \mathbb{Z}$. 
\end{proposition}

We then have that the class of $n$-flat complexes in $\chR$ is given by ${\rm dw}\,\widetilde{\mathsf{F}_n}$. Let us summarize in the following result the properties of $n$-absolutely pure and $n$-flat complexes.

\begin{proposition}\label{prop:basic_properties_complexes}
The following assertions hold:
\begin{enumerate}[(1)]
\item $({\rm dw}\,\widetilde{\mathsf{F}_n},{\rm dw}\,\widetilde{\mathsf{A}_n})$ is a perfect duality pair. 
\item $({\rm dw}\,\widetilde{\mathsf{F}_n},[{\rm dw}\,\widetilde{\mathsf{F}_n}]^{\perp_1})$ is a perfect cotorsion pair in $\chR$. In particular, every complex of $R^{\rm op}$-modules has an $n$-flat cover. 
\item ${\rm dw}\,\widetilde{\mathsf{F}_n}$ is closed under pure submodules, pure quotients and pure extensions. 
\item $({}^{\perp_1}[{\rm dw}\,\widetilde{\mathsf{A}_n}],{\rm dw}\,\widetilde{\mathsf{A}_n})$ is a cotorsion pair in $\Rch$ cogenerated by a set. In particular, every complex of $R$-modules has a $n$-absolutely pure special preenvelope. 
\end{enumerate}
\end{proposition}

\begin{proof}
Part (1) is a consequence of Proposition \ref{prop:basic_properties} (1) and \cite[Thm. 4.2.1 of the preprint version]{ZhaoPerez19}. Parts (2) and (3), on the other hand, follow from part (1) and \cite[Thm. 3.2]{Yang12}. Finally, we can obtain part (4) combining \cite[Prop. 4.4]{GillespieDegreewise} and Proposition \ref{prop:basic_properties} (4).
\end{proof}

\begin{remark}
Chain complex versions of \cite[Lem. 2.1, Thm. 2.1 (2) \& Prop. 3.1]{YangLiu11} are stated and proved in \cite[Prop. 3.5, Thm. 3.6 (2) \& Prop. 4.1]{SelvarajSaravanan16}. The observations in Remarks \ref{obs-initial}, \ref{rmk0}, \ref{obs-ideal} (1) concerning the proofs of these results given in \cite{YangLiu11} also apply to the proofs provided in \cite{SelvarajSaravanan16}. \newpage
\end{remark}

%%%%%%%%%%%%%%%%%%%%%%%%%%%%%%%%%%%%%
%%%%%%%%%%%%%%%%%%%%%%%%%%%%%%%%%%%%%
%%%%%%%%%%%%%%%%%%%%%%%%%%%%%%%%%%%%%
%%%%%%%%%%%%%%%%%%%%%%%%%%%%%%%%%%%%%

%\section{Some homotopical aspects from $n$-absolutely pure and $n$-flat modules}

%%%%%%%%%%%%%%%%%%%%%%%%%%%%%%%%%%%%%
%%%%%%%%%%%%%%%%%%%%%%%%%%%%%%%%%%%%%

%\subsection*{$n$-Absolutely pure and $n$-flat model structures}

%%%%%%%%%%%%%%%%%%%%%%%%%%%%%%%%%%%%%
%%%%%%%%%%%%%%%%%%%%%%%%%%%%%%%%%%%%%

%\subsection*{Induced model structures on chain complexes}

%%%%%%%%%%%%%%%%%%%%%%%%%%%%%%%%%%%%%
%%%%%%%%%%%%%%%%%%%%%%%%%%%%%%%%%%%%%

%\subsection*{Induced relative Gorenstein flat model structures}

%%%%%%%%%%%%%%%%%%%%%%%%%%%%%%%%%%%%%
%%%%%%%%%%%%%%%%%%%%%%%%%%%%%%%%%%%%%
%%%%%%%%%%%%%%%%%%%%%%%%%%%%%%%%%%%%%
%%%%%%%%%%%%%%%%%%%%%%%%%%%%%%%%%%%%%

%\section{Tracking citations}

%%%%%%%%%%%%%%%%%%%%%%%%%%%%%%%%%%%%%
%%%%%%%%%%%%%%%%%%%%%%%%%%%%%%%%%%%%%
%%%%%%%%%%%%%%%%%%%%%%%%%%%%%%%%%%%%%
%%%%%%%%%%%%%%%%%%%%%%%%%%%%%%%%%%%%%

\section*{Disclosure}

No AI models were used during the preparation of this paper. 

%%%%%%%%%%%%%%%%%%%%%%%%%%%%%%%%%%%%%
%%%%%%%%%%%%%%%%%%%%%%%%%%%%%%%%%%%%%
%%%%%%%%%%%%%%%%%%%%%%%%%%%%%%%%%%%%%
%%%%%%%%%%%%%%%%%%%%%%%%%%%%%%%%%%%%%

\section*{Funding}

The authors were supported by the following institutions:  ANII - Agencia Nacional de Investigación e Innovación, PEDECIBA - Programa de Desarrollo de las Ciencias Básicas, and CSIC - Comisión Sectorial de la Investigación Científica, Universidad de la República (via Apoyo a Proyectos de I+D, grant \# 22520240100305UD ``Condiciones de finitud en ideales de anillos'').

The second author was also supported by ANII via Math AMSUD project \# ANII-MOV-1012404 ``HHMA Homological and homotopical methods in algebra''.

%%%%%%%%%%%%%%%%%%%%%%%%%%%%%%%%%%%%%
%%%%%%%%%%%%%%%%%%%%%%%%%%%%%%%%%%%%%
%%%%%%%%%%%%%%%%%%%%%%%%%%%%%%%%%%%%%
%%%%%%%%%%%%%%%%%%%%%%%%%%%%%%%%%%%%%

\bibliographystyle{alpha}
\bibliography{biblio_almost}

\end{document}